\documentclass{amsart}
\usepackage[utf8]{inputenc}

\usepackage{amsmath}
\usepackage{latexsym,amsfonts,amssymb,mathrsfs}
\usepackage{mathdots}
\usepackage{color}
\usepackage[all,2cell,color]{xy}
\usepackage{enumitem} 
\usepackage{bbm}
\usepackage{xstring}
\usepackage[T1]{fontenc}
 \usepackage[utf8]{inputenc}
 \usepackage{array}
 \usepackage{blkarray} 
 \usepackage{mathtools} 
 \usepackage{longtable}
 \usepackage{multicol} 
 \usepackage{fdsymbol} 
 \usepackage{mathtools}
 \usepackage[left=30mm, right=35mm, top=35mm, bottom=40mm]{geometry}

\usepackage{tikz}
\usetikzlibrary{cd}
\usetikzlibrary{calc}
\usetikzlibrary{math}
\usetikzlibrary{arrows}
\usetikzlibrary{fit}
\usetikzlibrary{positioning}
\usetikzlibrary{decorations.pathmorphing, arrows.meta}
\tikzset{Rightarrow/.style={double equal sign distance,>={Implies},->},
triple/.style={-,preaction={draw,Rightarrow}},
quadruple/.style={preaction={draw,Rightarrow,shorten >=0pt},shorten >=1pt,-,double,double
distance=0.2pt}}

 \usepackage{hyperref}
\usepackage[capitalise]{cleveref}

\usetikzlibrary{positioning} 
\usetikzlibrary{decorations.pathreplacing}
\usetikzlibrary{arrows, decorations.markings} 

\UseAllTwocells
\usepackage[
textwidth=3cm,
textsize=small,
colorinlistoftodos]
{todonotes}

\newcolumntype{D}{>{\hfil$}p{1.2cm}<{$\hfil}}

\tikzset{%
    symbol/.style={%
        draw=none,
        every to/.append style={%
            edge node={node [sloped, allow upside down, auto=false]{$#1$}}}
    }
}

\tikzset{%
scalearrow/.style n args={3}{
  decoration={
    markings,
    mark=at position (1-#1)/2*\pgfdecoratedpathlength
      with {\coordinate (#2);},
    mark=at position (1+#1)/2*\pgfdecoratedpathlength
      with {\coordinate (#3);},
    },
  postaction=decorate,
  }
}

 \theoremstyle{plain}   
\newtheorem{thm}{Theorem}[subsection] 
\makeatletter\let\c@thm\c@thm\makeatother

\makeatletter\let\c@cor\c@thm\makeatother
\newtheorem{lem}{Lemma}[subsection]
\makeatletter\let\c@lem\c@thm\makeatother
\newtheorem{prop}{Proposition}[subsection]
\makeatletter\let\c@prop\c@thm\makeatother

\makeatletter\let\c@claim\c@thm\makeatother

\newtheorem{conjecture}{Conjecture}[subsection]
\makeatletter\let\c@conjecture\c@thm\makeatother

\theoremstyle{definition}

\newtheorem{defn}{Definition}[subsection]
\makeatletter\let\c@defn\c@thm\makeatother
\newtheorem{const}{Construction}[subsection]
\makeatletter\let\c@const\c@thm\makeatother
\newtheorem{notn}{Notation}[subsection]
\makeatletter\let\c@notn\c@thm\makeatother

\makeatletter\let\c@convention\c@thm\makeatother
\newtheorem{question}{Question}[subsection]
\makeatletter\let\c@convention\c@thm\makeatother

\theoremstyle{remark}

\newtheorem{rmk}{Remark}[subsection]
\makeatletter\let\c@rmk\c@thm\makeatother
\newtheorem{ex}{Example}[subsection]
\makeatletter\let\c@ex\c@thm\makeatother

\makeatletter\let\c@observation\c@thm\makeatother

\makeatletter\let\c@warning\c@thm\makeatother
\newtheorem{digression}{Digression}[subsection]
\makeatletter\let\c@digression\c@thm\makeatother

\makeatletter\let\c@answ\c@thm\makeatother

\makeatletter
\let\c@equation\c@thm
\numberwithin{equation}{subsection}
\makeatother

\crefname{lem}{Lemma}{Lemmas}
\crefname{thm}{Theorem}{Theorems}
\crefname{defn}{Definition}{Definitions}
\crefname{notn}{Notation}{Notations}
\crefname{const}{Construction}{Constructions}
\crefname{prop}{Proposition}{Propositions}
\crefname{rmk}{Remark}{Remarks}
\crefname{cor}{Corollary}{Corollaries}
\crefname{equation}{Display}{Displays}
\crefname{ex}{Example}{Examples}
\crefname{thmalph}{Theorem}{Theorems}
\crefname{answ}{Answer}{Answers}
\crefname{question}{Question}{Questions}

\newcommand{\cA}{\mathcal{A}}
\newcommand{\cB}{\mathcal{B}}
\newcommand{\cC}{\mathcal{C}}
\newcommand{\cD}{\mathcal{D}}
\newcommand{\cE}{\mathcal{E}}

\newcommand{\cI}{\mathcal{I}}

\newcommand{\cM}{\mathcal{M}}

\newcommand{\cP}{\mathcal{P}}
\newcommand{\cQ}{\mathcal{Q}}

\newcommand{\cS}{\mathcal{S}}

\newcommand{\cat}{\cC\!\mathit{at}}
\newcommand{\set}{\cS\!\mathit{et}}
\newcommand{\sset}{\mathit{s}\set}

\newcommand{\psh}[1]{\set^{#1^{\op}}}
\newcommand{\spsh}[1]{\sset^{#1^{\op}}}
\newcommand{\msset}{m\sset}
 \newcommand{\twocat}{2\cat}
 
 \newcommand{\omegacat}{\omega\cat}

 \newcommand{\comp}{\circ}

\DeclareMathOperator{\Hom}{Hom}
\DeclareMathOperator{\Map}{Map} 
\DeclareMathOperator{\Mor}{Mor}
\DeclareMathOperator{\Ob}{Ob}

\DeclareMathOperator{\core}{core}

\DeclareMathOperator{\colim}{colim}

\DeclareMathOperator{\id}{id}

\newcommand{\aamalg}[1]{\underset{#1}{{\amalg}}} 
\newcommand{\tttimes}[1]{\underset{{#1}}{\times}} 
\newcommand{\ootimes}[1]{\overset{\infty}{\underset{{#1}}{\times}}} 
\DeclareMathOperator{\op}{op}

\DeclareMathOperator{\pr}{pr}

\newcommand{\chaos}{\mathrm{ch}}
\newcommand{\roughtrunc}[1]{\tau_{\leq #1}^{\text{b}}}
\newcommand{\inttrunc}[1]{\tau_{\leq #1}^{\text{i}}}
\newcommand{\term}{*}

\makeatletter
\def\elseif{\elseif}\def\endif{\endif}
\def\ifnumcase#1#2{%
    \expandafter\ifx\expandafter\elseif\@car#2\@nil\expandafter\@firstoftwo
    \else
        \expandafter\ifx\expandafter\endif\@car#2\@nil\expandafter\expandafter\expandafter\@gobbletwo
        \else\expandafter\expandafter\expandafter\@secondoftwo
        \fi
    \fi
        \idto@endif
        {\if@eqin#2=\@nil{\if@dimwitheq{#1}#2\@nil}{\ifdim#1pt#2pt }
            \expandafter\@firstoftwo\else\expandafter\@secondoftwo\fi\exec@arg{\ifnumcase@i{#1}}%
        }%
    }
\def\ifnumcase@i#1#2{\ifnumcase{#1}}
\def\if@eqin#1=#2\@nil{%
    \ifx\@empty#2\@empty\expandafter\@secondoftwo
    \else
        \ifx\@empty#1\@empty\expandafter\expandafter\expandafter\@secondoftwo
        \else\expandafter\expandafter\expandafter\@firstoftwo
        \fi
    \fi}
\def\if@dimwitheq#1#2=#3\@nil{\unless\ifdim#1pt\if<#2>\else<\fi#3pt }
\def\exec@arg#1#2\endif{#1}
\def\idto@endif#1\endif{#1}
\def\gobto@endif#1\endif{}
\makeatother
\newcommand{\eq}[1]{
\IfInteger{#1}{
\ifnumcase{#1}
    {=1}{\mathcal{I}}
    {=2}{\mathcal{E}}
\elseif
  {{#1}\mathcal{E}}
  \endif
}
{
{{#1}\mathcal{E}}
}
}

\newcommand{\coheq}[1]{
\IfInteger{#1}{
\ifnumcase{#1}
    {=1}{\mathcal{I}}
    {=2}{\mathcal{E}^{\mathrm{adj}}}
\elseif
  {{#1}\mathcal{E}^{\mathrm{adj}}}
\endif
}{
{{#1}\mathcal{E}^{\mathrm{adj}}}
}
}

\AtBeginDocument{%
   \def\MR#1{}
}

\author{Viktoriya Ozornova}
\address{Max Planck Institute for Mathematics, Bonn, Germany}
\email{viktoriya.ozornova@mpim-bonn.mpg.de}

\author{Martina Rovelli}
\address{Department of Mathematics and Statistics,
University of Massachusetts, 
Amherst,
USA
}
\email{mrovelli@umass.edu}

\subjclass[2020]{18N10; 18N30; 18N60; 18N65; 18N40}

\title{What is an equivalence in a higher category?}

\begin{document}

\maketitle

\begin{abstract}
The purpose of this survey is to present in a uniform way the notion of equivalence between strict $n$-categories or $(\infty,n)$-categories, and inside a strict $(n+1)$-category or $(\infty,n+1)$-category.
\end{abstract}

\tableofcontents

\section*{Introduction}

Many theorems in mathematics are about identifying two types  of mathematical objects of interest that present themselves as seemingly different. While it is tempting to state such results by saying that two classes of mathematical objects are \emph{equal}, in practice -- rather than \emph{equality} -- the correct mathematical notion of sameness that encapsulates many such correspondences is that of a \emph{bijection}. For instance, here are some well-known examples of correspondences in math that are expressed by means of suitable bijections:
\begin{enumerate}[leftmargin=*]
   \item \emph{Galois correspondence}: Given a finite field extension $E/F$, there's a bijection between the set of intermediate field extensions of $E/F$ and the set of subgroups of the Galois group of $E/F$.
    \item \emph{Classification of covering spaces}: Given a nice pointed space $X$, there's a bijection between the set of isomorphism classes of pointed covering spaces of $X$ and the set of subgroups of the fundamental group of $X$.
    \item \emph{Stone duality}:
     There's a bijection between the large set of Boolean algebras and the large set of Boolean rings.
\end{enumerate}

We may notice at this point that each of the aforementioned examples is in fact the shadow of a stronger statement. For instance, in the Galois correspondence it also holds that every inclusion of intermediate field extensions of $E/F$ correspond to a (reverse) inclusion of the corresponding subgroups of the Galois group of $E/F$.

More generally, it if often the case that the mathematical objects of interest come with a relevant notion of morphisms between them, and naturally assemble into a \emph{category}. While a set only consists of elements, a category consists of elements called objects, as well as arrows or morphisms between objects, together with a unital and associative composition law.

It is then desirable to upgrade the correspondence statements to suitable notions of sameness for the corresponding categories. At a first glance, one could hope to generalize the notion of bijection of sets to that of \emph{isomorphism of categories}. Roughly speaking, an isomorphism of categories from $\cA$ to $\cB$ consists of functorial assignments $F\colon\cA\to\cB$ and $G\colon\cB\to\cA$, and is designed so that for all objects $a$ of $\cA$ one has that $a$ is \emph{equal} to $G(F(a))$, and similarly for $\cB$.

Here are some correspondences that can be expressed by means of isomorphisms of categories.

\begin{enumerate}[leftmargin=*]
\item \emph{Categorified Galois correspondence}: Given a finite field extension $E/F$, there's a contravariant isomorphism between the category of intermediate field extensions of $E/F$ and inclusions and the category of subgroups of the Galois group of $E/F$ and inclusions.
\item \emph{Universal property of group algebras}: Given a group $G$ and a field $K$, there is an isomorphism between the category of modules over the group algebra $K[G]$ and the category of representations of $G$. In a similar vein, there's an isomorphism between the category of abelian groups and the category of modules over $\mathbb Z$.
\end{enumerate}

While extremely natural and intuitive, the notion of isomorphism of categories is too strict, and fails at encompassing the majority of mathematical correspondences of interest. Roughly speaking, the issue is that in many situations one has functorial assignments $F\colon\cA\to\cB$ and $G\colon\cB\to\cA$, but 
given an object $a$ of $\cA$ one has that $a$ is not equal, but rather \emph{isomorphic} to $G(F(a))$ in the category $\cA$, and similarly for $\cB$. This is the defining property for an equivalence of categories, as opposed to that of an isomorphism of categories.

Here are some classical correspondences that can be expressed by means of equivalences of categories which are \emph{not} isomorphisms of categories.

\begin{enumerate}[leftmargin=*]
  \item \emph{Zariski duality} \cite[\textsection2]{GoertzWedhorn}:
  In algebraic geometry there is a contravariant equivalence between the category of affine schemes and the category of commutative rings.
      \item \emph{Categorified classification of covering spaces} \cite[\textsection3]{tomDieckAT}: 
      Given a nice 
      space $X$, there is an equivalence between the category of covering spaces over $X$ and the category of sets with an action of the fundamental groupoid of $X$.    
       \item \emph{Morita theory}:
       Given a commutative ring $R$,
       there is an equivalence between the category of modules over $R$ and the category of modules over $\mathrm{Mat}_{n\times n}R$. Also, given commutative rings $R$ and $S$, there is an equivalence between the product of the categories of modules over $R$ and $S$, and the category of modules over $R\times S$.
    \item \emph{Gelfand duality} \cite{Gelfand}: In functional analysis there is a contravariant equivalence between the category of commutative $C^*$-algebras and the category of compact Hausdorff topological spaces.
  
\item \emph{Dold--Kan correspondence} \cite[\textsection III]{GoerssJardine}:
There's an equivalence between the category of simplicial abelian groups and the category of non-negatively graded chain complexes.
     \item \emph{$1$-categorical straightening-unstraightening} \cite{LR}: Given a category $\cC$, there is an equivalence between the category of presheaves on $\cC$ and the category of discrete fibrations over $\cC$.
    \item \emph{Low dimensional cobordism hypothesis} \cite{Kock}: In mathematical physics, given a field $K$ there is an equivalence between the category of vector space valued $1$-dimensional topological quantum field theories over $K$ and the category of finite-dimensional $K$-vector spaces, and an equivalence between the category of $K$-vector space valued $2$-dimensional topological quantum field theories and the category of commutative Frobenius $K$-algebras.
\end{enumerate}

But, once again, too many phenomena of interest cannot be formalized via the notion of equivalence and even more fundamentally via the notion of category itself. 
Indeed, many mathematical objects of interest naturally assemble into something more complex than an ordinary category, such as an $n$-category, an $\infty$-category, or more generally an $(\infty,n)$-category.

The differences amongst these notions arise essentially from choosing different combinations of two parameters: whether there are non-invertible morphisms in dimension higher than $1$, and whether the axioms hold on the nose or up to a higher morphism. Roughly speaking, the differences amongst these types of higher categories can be summarized by the following table
\begin{center}
\begin{tabular}{|c|cc|}
\hline
     &strict axioms& weak axioms\\
     \hline
     invertibility of morphisms in dim $>0$& set & $\infty$-groupoid\\
 invertibility of morphisms in dim $>1$& category & $\infty$-category\\
   invertibility of morphisms in dim $>n$& $n$-category & $(\infty,n)$-category\\ 
    \hline
\end{tabular}
\end{center}
where the first and second row are special cases of the last one when taking $n=0,1$.

Notable examples of higher categories that cannot be modelled by ordinary categories include:
\begin{enumerate}[leftmargin=*]
   \item Versions of the $\infty$-category of spaces, of rational spaces, or of spectra;
    \item Versions of the $\infty$-category of chain complexes, and of DG-algebras;
 \item The $2$-category of categories and the $(\infty,2)$-category of $\infty$-categories;
    \item Versions of the $(\infty,n)$-category of cobordisms \cite{bd,luriecobordism,AyalaThesis,CalaqueScheimbauer, AFcobordism};
    \item Versions of the $(\infty,2)$-category of spans or correspondences of $\infty$-groupoids, or more generally inside a monoidal $(\infty,2)$-category \cite{luriecobordism,GaitsgoryRozenblyumVolI,HaugsengSpan,dk,HarpazSpan};
    \item Versions of the $(\infty,n+1)$-Morita category of $\mathbb E_n$-algebras \cite{JFSMorita,HaugsengMorita}.
\end{enumerate}

And here are several correspondences that need to be expressed as $(\infty,n)$-equivalences of $(\infty,n)$-categories which are \emph{not} equivalences of ordinary categories.

\begin{enumerate}[leftmargin=*]

 \item \emph{Algebraic models in rational and $p$-adic homotopy theory} \cite{QuillenRational, Sullivan, Mandell}: There's an $(\infty,1)$-equivalence between the $\infty$-category of rational spaces and the $\infty$-category of reduced
differential graded Lie algebras over $\mathbb Q$ and also to the homotopy theory of
$2$-reduced differential graded cocommutative coalgebras over $\mathbb Q$. Sullivan's algebraic model defines an $(\infty,1)$-equivalence between the $\infty$-category of rational spaces and the $\infty$-category of Sullivan algebras. In a similar vein, given any prime number $p$, there's an $(\infty,1)$-equivalence between the $\infty$-category of $p$-adic spaces and an appropriate $\infty$-category of $\mathbb E_{\infty}$-algebras.

\item \emph{$2$- and $(\infty,1)$-categorical straightening-unstraightening} \cite{Grothendieck,htt,NuitenStraightening}: Given a $2$-category $\cC$, there is a $2$-equivalence of $2$-categories between the $2$-category of categorical pseudo-presheaves over $\cC$ and the $2$-category of Grothendieck fibrations over $\cC$. Also, given a $\infty$- (or more generally $(\infty,n)$-)category $\cC$, there is an $(\infty,1)$-equivalence between the $\infty$-category of $\infty$- (or more generally $(\infty,n)$-)categorical presheaves and the $\infty$-category of cartesian fibrations over $\cC$.

\item \emph{Fully extended cobordism hypothesis} \cite{bd,luriecobordism,FreedCobSurvey,AFcobordism}: Given a monoidal $(\infty,n)$-category $\cC$, evaluation at a point defines an $(\infty,0)$-equivalence between the $\infty$-groupoid of fully extended $n$-dimensional topological quantum field theories valued in $\cC$ and the $\infty$-groupoid of fully dualizable objects in $\cC$.

\item \emph{Homological mirror symmetry} \cite{KontsevichMirror,KontsevichSoibelmanTorus,GanatraPardonShende} Given a symplectic manifold $X$, there is a conjectural equivalence of $A_\infty$-categories -- which are certain stable $\infty$-categories -- between an appropriate Fukaya category of $X$ and the 
derived category of coherent sheaves on the mirror variety of $X$.

\item \emph{Stable homotopy hypothesis} \cite{JO,GJO,GH,MOPSV}: There is an $(\infty,1)$-equivalence between the $(\infty,1)$-category of Picard $n$-categories and the $(\infty,1)$-category of stable $n$-types.

 \item \emph{$\mathbb E_n$-algebras and homology theories} \cite{AF1}:
Factorization homology defines an $(\infty,1)$-equivalence between an appropriate $\infty$-category of $\mathbb E_n$-algebras and an appropriate $\infty$-category of excisive homology theories.

\item \emph{Derived categories of quasi-coherent sheaves} \cite{MathewMeierAffineness}: Given a derived stack $X$ satisfying a suitable quasi-affineness condition, there is a monoidal $\infty$-equivalence (hence an $(\infty,2)$-equivalence) between the monoidal $\infty$-category of quasi-coherent sheaves and the monoidal $\infty$-category of modules over the global sections of $X$.

\item \emph{$\infty$-categorical Dold--Kan correspondence} \cite{DyckerhoffDK}: There's an $(\infty,1)$-equivalence between an appropriate $\infty$-category of simplicial stable $\infty$-categories and an appropriate $\infty$-category of connective chain complexes of stable $\infty$-categories.

\item \emph{Local geometric Langlands duality} \cite{ArinkinGaitsgory}: 
Given a connected reductive group $G$, there is conjecturally an $(\infty,1)$-equivalence between an appropriate $\infty$-category of $D$-modules on the moduli stack of $G$-bundles and an appropriate $\infty$-category of ind-coherent sheaves on a DG-stack of local system on the dual $\breve{G}$.
\end{enumerate}

So far we have only mentioned the \emph{external viewpoint} on equivalences, meaning the notion of $(\infty,n)$-equivalence \emph{between} $(\infty,n)$-categories. But there is also an \emph{internal viewpoint} on equivalences, which is about understanding the notion of an $(\infty,n)$-equivalence of objects \emph{inside} an $(\infty,n+1)$-category.
For instance, bijections of sets can be seen as isomorphisms in the category of sets and equivalences of categories can 
be understood as equivalences in the $2$-category of categories, and one can define more generally isomorphisms in any category and equivalences in any $2$-category. Even more generally, one can define $n$-equivalences inside an $(n+1)$-category and $(\infty,n)$-equivalences inside an $(\infty,n+1)$-category. This is the correct notion of sameness for objects inside a generic $(\infty,n+1)$-category.

The purpose of this paper is to survey the various notions of equivalence between and inside the various higher categorical structures: sets, categories, $n$-categories and $(\infty,n)$-categories, treating both the internal and external viewpoint each time. We also discuss some model categorical tools to help interpret an $(\infty,n+1)$ presented by a model at the model categorical level.

Note: This paper targets someone who is familiar with the ordinary category theory language. It starts by recalling elementary notions, such as the notion of a category, or isomorphism in a category, or even the notion of a bijection of sets. This is not done thinking that the reader is learning this for the first time, but instead aiming to stress one or more viewpoints that are suitable for generalizations for more complicated contexts.

\addtocontents{toc}{\protect\setcounter{tocdepth}{1}}
\subsection*{Acknowledgements}
We are thankful for insightful discussions with Dimitri Ara, Nick Gurski, Simon Henry, F\'elix Loubaton,  Lennart Meier,  Nima Rasekh, and Alex Rice. The second author is grateful for support from the National Science Foundation under Grant No. DMS-2203915.

\addtocontents{toc}{\protect\setcounter{tocdepth}{2}}

\section{Equivalences of and inside strict higher categories}
\label{EquiStrict}


\subsection{Isomorphism of sets and inside a category}

\subsubsection{Bijections of sets}

We recall the notion of bijection of sets, stressing a (perhaps odd) viewpoint that helps prepare the ground for the viewpoint of sets being $0$-categories. In virtue of this, we introduce the following auxiliary notation:

\begin{notn}
For a set $\cA$ and $a,a' \in \cA$, we denote
    \[\cA(a,a')=\left\{\begin{array}{lll}
         \{\term\} & \text{ if }a=a' \\
         \varnothing & \text{ if }a\neq a'.
    \end{array}\right.\]
    \end{notn}

One way to define bijections is in terms of injectivity and surjectivity:

\begin{defn}[Bijection of sets]
\label{defbijsets}
A function $F\colon \cA\to \cB$ between sets $\cA$ and $\cB$ is a \emph{bijection}, and we write $F\colon \cA\cong \cB$, if 
\begin{enumerate}[leftmargin=*]
    \item $F$ is \emph{surjective}, meaning that for every element $b\in\cB$ there exists an element $a\text{ in } \cA$ and equalities of elements
    \[b=Fa\text{ in } \cB;\]
    \item $F$ is \emph{injective}, meaning that for every elements  $a,a'\text{ in } \cA$ there are equalities
    \[\cA(a,a')=\cB(Fa,Fa')\in\left\{ \{\term\} ,\varnothing\right\}.\]
\end{enumerate}
Two sets $\cA$ and $\cB$ are \emph{in bijection}, and we write $\cA\cong \cB$, if there is a bijection between them.
\end{defn}

There is a characterization of bijections in terms the existence of an inverse function.

\begin{prop}[Isomorphism in the category $\set$]
A function $F\colon \cA\to \cB$ between sets is a bijection if and only if it is \emph{invertible}; i.e., if there exists a function $G\colon \cB\to \cA$ and equalities
    \[G\circ F=\id_\cA \text{ in } \set(\cA,\cA)\text{ and }F\circ G=\id_\cB\text{ in } \set(\cB,\cB).\]
\end{prop}

This viewpoint will be taken to be the notion of isomorphism in the category of sets, and is generalizable to define isomorphisms in any category $\cC$.

\subsubsection{Categories}

Recall from e.g.~\cite[\textsection I.2]{MacLane}
that a (\emph{small}) \emph{category}
$\cC$ consists of
\begin{itemize}[leftmargin=*]
    \item a set of \emph{objects} $\Ob\cC$,
    \item for every $c,c'\in\Ob\cC$ a \emph{hom-set} $\cC(c,c')$,
    \item an identity operator function for every
$c\in\Ob\cC$ 
\[\id_c\colon\{\term\}\to\cC(c,c)\quad\quad\term \mapsto \id_c\]
\item a \emph{composition operator} function for every $c,c',c''\in\Ob\cC$
\[\circ\colon\cC(c,c')\times\cC(c',c'')\to\cC(c,c'')\quad\quad(f,g) \mapsto g\circ f\]
\end{itemize}
satisfying for every $c,c',c'',c'''\in\Ob\cC$ the \emph{associativity axiom}, given by the commutativity of the diagram of sets and functions 
\begin{equation}
\label{associativity}
\begin{tikzcd}[column sep=3cm]
\cC(c,c')\times\cC(c',c'')\times\cC(c'',c''')\arrow[r,"{\id_{\cC(c,c')}\times\circ}"]\arrow[d, "{\circ\times\id_{\cC(c'',c''')}}" swap]&\cC(c,c')\times\cC(c',c''')\arrow[d,"\circ"]\\
\cC(c,c'')\times\cC(c'',c''')\arrow[r,"\circ" swap]&\cC(c,c''')
\end{tikzcd}
\end{equation}
for every $c,c'\in\Ob\cC$ the \emph{unitality axiom}, given by the commutativity of the diagram of sets and functions
\begin{equation}
\label{unitality}
\begin{tikzcd}
\{*\}\times\cC(c,c')\arrow[dr, "\pr_2"] \arrow[d, "{\id_c \times \id_{\cC(c,c')}}"swap]&\\
\cC(c,c)\times\cC(c,c')\arrow[r,"\circ" swap]&\cC(c,c')
\end{tikzcd}
\mbox{ and }\quad
\begin{tikzcd}
\cC(c,c')\times \{*\}\arrow[dr, "\pr_1"]\arrow[d, "{\id_{\cC(c,c')}\times \id_{c'}}" swap]&\\
\cC(c,c')\times\cC(c',c')\arrow[r,"\circ" swap]&\cC(c,c')
\end{tikzcd}
\end{equation}

In particular, the data of a category
$\cC$ also determines a set $\Mor\cC$ of \emph{morphisms}, given by the union of all hom-sets
\[\Mor\cC\coloneqq\coprod_{c,c'\in\Ob\cC}\cC(c,c'),\]
as well as several operators:
\begin{itemize}[leftmargin=*]
    \item \emph{source} and \emph{target} operators $s,t\colon \Mor\cC \to \Ob\cC$,
    \item \emph{identity} operator
    $\id \colon \Ob\cC \to \Mor\cC$
    \item \emph{composition} operators
    $\comp \colon \Mor\cC \times_{\Ob\cC} \Mor\cC \to \Mor\cC$, defined all pairs of morphisms $(f,g)$ such that $s(g) = t(f)$.
\end{itemize}

Given categories $\cC$ and $\cD$, a \emph{functor} $F\colon\cC\to\cD$ consists of
\begin{itemize}[leftmargin=*]
    \item a function on objects
\[F\colon\Ob\cC\to\Ob\cD\]
\item and a function on hom-sets for every $c,c'\in\Ob\cC$
\[F_{c,c'}\colon\cC(c,c')\to\cD(Fc,Fc'),\]
\end{itemize}
satisfying the \emph{functorial properties}; namely, the compatibility with composition, given by the commutativity for every $c,c',c''\in\Ob\cC$ of the diagram
\begin{equation}
    \label{functorial1}
\begin{tikzcd}
\cC(c,c')\times \cC(c',c'')\arrow[r]\arrow[d, "\circ"']&\cD(Fc,Fc')\times \cD(Fc',Fc'')\arrow[d, "\circ"]\\
\cC(c,c'')\arrow[r]&\cD(Fc,Fc'')
\end{tikzcd}
\end{equation}
and the compatibility with identity, given by the commutativity for every $c\in\Ob\cC$ of the diagram
\begin{equation}
    \label{functorial2}
\begin{tikzcd}
\{\term\}\arrow[r]\arrow[d]&\{\term\}\arrow[d]\\
\cC(c,c)\arrow[r]&\cD(Fc,Fc)
\end{tikzcd}
\end{equation}

For instance, $\set$ is the large\footnote{A \emph{large category} $\cC$ is defined analogously to a small category, with the difference that $\Ob\cC$ is allowed to be a class, as opposed to a set. The formalization of categories of different sizes can be made precise by making use of \emph{Grothendieck universes} (cf.~\cite[Expos\'e I, Appendice II]{SGA4}).} category of sets and functions, with usual composition of functions and identity functions of sets.

\subsubsection{Isomorphisms in a category}
The following definition recovers that of bijection of sets when read in the case of the category $\cC=\set$.

\begin{defn}[Isomorphism in a category $\cC$]
\label{defisocat}
A morphism $F\colon A\to B$ between objects $A$ and $B$ in a category $\cC$ is an \emph{isomorphism}, and we write $F\colon A\cong B$, if there exists a morphism $G\colon B\to A$ in $\cC$ and equalities
    \[G\circ F=\id_A \text{ in } \cC(A,A)\text{ and }F\circ G=\id_{B}\text{ in } \cC(B,B)\]
Two objects $A$ and $B$ of a $1$-category $\cC$ are \emph{isomorphic}, and we write $A\cong B$, if there is an isomorphism between them in $\cC$.
\end{defn}

Isomorphisms can also be described in terms of the existence of inverses on each side (a viewpoint that will inspire \cref{constructionJ}):

\begin{rmk}
\label{IsoTwoInverses}
A morphism $F\colon A\to B$ in a category $\cC$ is an isomorphism if and only if there exist morphisms $G,G'\colon B\to A$ such that
\[F\circ G=\id_B\quad\text{ and }\quad G'\circ F=\id_A,\]
and this formulation was taken as one the first definitions of isomorphism in a category in \cite[\textsection I.1]{EilenbergMacLaneNatEqui}.
\end{rmk}

It is easily verified that the relation of being isomorphic for sets or more generally objects in a category $\cC$ is an equivalence relation.

\subsubsection{The walking isomorphism}

One can identify the universal shape $\cI$ that encodes the notion of an isomorphism, as we briefly recall after introduce the auxiliary category $\cP$.

\begin{const}
\label{constP}
We denote by $\cP$ the \emph{free category on two morphisms with opposite directions} $f\colon a\to b$ and $g\colon  b\to a$. This is obtained by gluing $f$ and $g$ ``head-to-tail'', and generating all possible compositions.
Explicitly, the set of objects is $\Ob\cP=\{a, b\}$ and the hom-sets are
\[\cP(a,a)=\{(g\circ f)^n\ |\ n\ge0\}\quad\cP(a, b)=\{f\circ(g\circ f)^n\ |\ n\ge0\}\]
\[\cP( b, a)=\{g\circ(f\circ g)^n\ |\ n\ge0\}\quad\cP( b, b)=\{(f\circ g)^n\ |\ n\ge0\}.\]
The category $\cP$ can be understood as the pushout of categories 
\[
\begin{tikzcd}[column sep=3.15cm]
\partial\cC_1\amalg\partial\cC_1^{\op}\arrow[r,""]\arrow[d,hook] \arrow[dr, phantom, "\pushout", very near end, yshift=0.2cm, xshift=0.1cm]&\cC_0\amalg\cC_0\arrow[d]\\
\cC_1\amalg\cC_1^{\op}\arrow[r]&\cP
\end{tikzcd}
\]
Here,
\begin{itemize}[leftmargin=*]
    \item $\cC_0$ denotes the walking object category, which consists of one object and no non-identity morphisms;
    \item $\cC_1$ and $\cC_1^{\op}$ denote the walking morphism category and its opposite, which consist of one morphism -- $f\colon a\to b$ and $g\colon b\to a$ respectively -- between different objects and no other non-identity morphisms;
    \item $\partial\cC_1=\partial\cC_1^{\op}$ denotes the boundary of $\cC_1$, consisting of two objects $a$ and $b$ and no non-identity morphisms.
\end{itemize} 
The vertical map is the canonical inclusion and the two cocomponents of the top horizontal map are the two distinct isomorphisms $\partial\cC_1\cong\cC_0\amalg\cC_0$.
\end{const}

\begin{const}[The walking isomorphism]
\label{constI}
We denote by $\cI$ the \emph{walking isomorphism} category, which is obtained from the category $\cP$ by imposing the relations
\[f\circ g=\id_{ b}\quad\text{ and }\quad g\circ f=\id_{a}.\]
The set of objects is $\Ob\cI=\{a, b\}$ and the hom-sets are
\[\cI(a,a)=\{\id_a\}\quad\cI( a, b)=\{f\}\quad\cI( b, b)=\{\id_b\}\quad\cI( b, a)=\{g\}.\]
The category $\cI$ can be understood as the pushout of categories
\[
\begin{tikzcd}[column sep=3.15cm]
\partial\cC_2\amalg\partial\cC_2\arrow[r,"{[(f\circ g, \id_ b), (g\circ f, \id_{a})]}"]\arrow[d,"" swap]
\arrow[dr, phantom, "\pushout", very near end, yshift=0.2cm, xshift=0.1cm]&\cP\arrow[d]\\
\cC_1\amalg\cC_1\arrow[r]&\cI
\end{tikzcd}
\]
where $\partial\cC_2$ denotes the category consisting of two parallel morphisms and both the cocomponents of the vertical map are the canonical \emph{folding} map $\partial\cC_2\to\cC_1$ that identifies the two morphisms.
\end{const}

By design, the walking isomorphism $\cI$ detects isomorphisms in the following sense:

\begin{rmk}
\label{Idetects}
A morphism $F\colon A\to B$ in a category $\cC$ is an \emph{isomorphism} if and only if the functor $ F\colon\cC_1\to\cC$ determined by $F$ extends to a functor $\cI\to\cC$ along the inclusion $\cC_1\hookrightarrow \cI$, i.e., if
there is a solution to the lifting problem of categories
\[\begin{tikzcd}
{\cC_1}\arrow[r," F"]\arrow[d,hook, "f" swap]&\cC\\
\cI\arrow[ru,dashed]
\end{tikzcd}\]
\end{rmk}

\subsubsection{The fundamental set of a category}

The isomorphism relation for objects in a category can also be tested as equality relation in a suitable quotient set.

\begin{defn}
Given a category $\cC$, the \emph{fundamental set} of $\cC$ is the set
\[\Pi_0\cC\coloneqq\Ob\cC/\cong\]
of the isomorphism classes of objects in $\cC$.
This defines a product-preserving functor $\Pi_0\colon\cat\to\set$.
\end{defn}

Essentially by definition, we get:

\begin{rmk}
\label{FundCatEqui1cat}
Given a category $\cC$ and $A,B\in\Ob\cC$ the following are equivalent.
\begin{enumerate}[leftmargin=*]
    \item There is an isomorphism
    \[A\cong B\text{ in the category }\cC.\]
    \item There is an equality
    \[A= B\text{ in the fundamental set }\Pi_0\cC.\]
    \end{enumerate}
\end{rmk}

\subsection{Equivalence of categories and inside a $2$-category}

\subsubsection{Equivalences of categories}

The notion of equivalence of categories is a classical one, occurring at least as early as \cite[\textsection I.8]{GabrielAbeliennes} and discussed e.g.\ in \cite[\textsection IV.4]{MacLane}.

\begin{defn}[Equivalence of categories]
\label{defequiofcat}
A functor $F\colon \cA\to \cB$ between categories is an \emph{equivalence}, and we write $F\colon \cA\simeq \cB$, if
\begin{enumerate}[leftmargin=*] 
    \item the functor $F$ is \emph{surjective on objects up to isomorphism}, meaning that for every object $b\in\Ob\cB$ there exists an object $a\in\Ob\cA$ and an isomorphism
    \[b\cong Fa\text{ in the category } \cB.\] 
    \item the functor $F$ is \emph{fully faithful} or a \emph{hom-wise isomorphism}, meaning that for all objects $a,a'\text{ in } \cA$, the functor $F$ induces an isomorphism
    \[F_{a,a'}\colon \cA(a,a')\cong \cB(Fa,Fa')\text{ of sets}.\]
\end{enumerate}
Two categories $\cA$ and $\cB$ are \emph{equivalent}, and we write $\cA\simeq \cB$, if there exists an equivalence of categories between them.
\end{defn}

\begin{prop}
\label{Icontractible}
The walking isomorphism category $\cI$ is \emph{contractible}, meaning that there is an equivalence of categories
\[\cI\simeq\cC_0.\]
\end{prop}

\begin{proof}
Consider the unique functor $\varphi\colon\cI\to\cC_0$. There is an isomorphism
(in fact an equality)
\[\term\cong \varphi(a)\text{ in }\cC_0,\]
and furthermore $\varphi$ induces an isomorphism
\[\varphi_{a,a}\colon\cI(a,a)=\{\id_a\}\cong\{\id_*\}
=\cC_0(*,*)=\cC_0(\varphi(a),\varphi(a))\text{ in }\set.\]
Similarly, the functions $\varphi_{b,b}$, $\varphi_{a,b}$ and $\varphi_{b,a}$ are isomorphisms. So there's an equivalence $\varphi\colon\cI\simeq\cC_0$, as desired.
\end{proof}

\begin{rmk}
In the context of the canonical model structure on the category $\cat$ (see e.g.~\cite{RezkCat} or \cite{JoyalVolumeII} for a description), the object $\cI$ is a cylinder for the terminal object in the sense of \cite[Definition 1.2.4]{hovey}. Meaning, it comes with a factorization
\[\cC_0\amalg\cC_0\hookrightarrow\cI\stackrel\simeq\longrightarrow\cC_0\]
of the folding map of $\cC_0$ into a cofibration followed by a weak equivalence (in fact acyclic fibration). Note that $\cI$ is not the unique cylinder, and shouldn't be expected to be. In fact, for every $n\ge0$ the connected groupoid on $n$ elements, obtained as the iterated pushout $\cI\amalg_{\cC_0}\dots\amalg_{\cC_0}\cI$ of $n$-copies of $\cI$ along objects,
is also a  cylinder on the terminal object in the canonical model structure on $\cat$.
Every other cylinder object would be qualified to detect isomorphisms in a category as in \cref{Idetects}.
\end{rmk}

There is a characterization of equivalences between categories
in terms the existence of a ``weak'' inverse. See e.g.~\cite[Proposition 12]{GabrielAbeliennes} or \cite[\textsection IV.4]{MacLane} for a proof.

\begin{prop}
A functor $F\colon \cA\to \cB$ between categories is an \emph{equivalence} if
and only if there exists a functor $G\colon \cB\to \cA$ and isomorphisms
    \[G\circ F\cong \id_\cA\text{ in } \cat(\cA,\cA)\text{ and }  F\circ G\cong\id_\cB\text{ in } \cat(\cB,\cB).\]
\end{prop}

This viewpoint can be taken to be the notion of equivalences in the $2$-category of categories, and is generalizable to define equivalences in any $2$-category $\cC$.

\subsubsection{$2$-categories}

Recall from e.g.~\cite[\textsection XII.3]{MacLane}
that a \emph{$2$-category}\footnote{The notion of $2$-category also occurs in \cite[\textsection I.5]{EilenbergKelly} (under the name of \emph{hypercategory}), based on ideas from \cite{EhresmannDouble,BenabouRelative}.} $\cC$ consists of \begin{itemize}[leftmargin=*]
    \item a \emph{set of objects} $\Ob\cC$;
    \item for every $c,c'\in\Ob\cC$, a \emph{hom-category} $\cC(c,c')$;
    \item for every $c\in\Ob\cC$, an \emph{identity} functor
\[\id_c\colon\{*\}\to\cC(c,c)\quad\quad* \mapsto \id_c\]
    \item for every $c,c',c''\in\Ob\cC$ a \emph{composition} functor 
\[\circ\colon\cC(c,c')\times\cC(c',c'')\to\cC(c,c'')\quad\quad(f,g) \mapsto g\circ f;\]
\end{itemize}
satisfying the associativity axiom \eqref{associativity} and the unitality axiom \eqref{unitality} read as diagrams of categories and functors, as opposed to sets and functions.

Given categories $\cC$ and $\cD$, a \emph{$2$-functor} $F\colon\cC\to\cD$ consists of \begin{itemize}[leftmargin=*]
    \item a function on objects
\[F\colon\Ob\cC\to\Ob\cD\]
\item for every $c,c'\in\Ob\cC$, a functor on hom-categories
\[F_{c,c'}\colon\cC(c,c')\to\cD(Fc,Fc');\]
\end{itemize}
satisfying the \emph{functorial properties} as \eqref{functorial1} and \eqref{functorial2} read as diagrams of categories and functors, as opposed to sets and functions.

In particular, the data of a $2$-category
$\cC$ determines a set $1\Mor\cC$ of \emph{$1$-morphisms} and a set $2\Mor\cC$ of \emph{$2$-morphisms} given by
\[1\Mor\cC\coloneqq\coprod_{c,c'\in\Ob\cC}\Ob(\cC(c,c'))\quad\text{ and }\quad2\Mor\cC\coloneqq\coprod_{c,c'\in\Ob\cC}\Mor(\cC(c,c')),\]
as well as several operators:
\begin{itemize}[leftmargin=*]
    \item \emph{source} and \emph{target} operators \[s_0, t_0 \colon \Mor\cC \to \Ob\cC,\quad s_1, t_1 \colon 2\Mor\cC \to \Mor\cC,\quad s_0, t_0 \colon 2\Mor\cC \to \Ob\cC\]
    \item \emph{identity} operators
    \[\id_1 \colon \Ob\cC \to \Mor\cC,\quad\id_2 \colon \Ob\cC \to 2\Mor\cC,\quad\id_2 \colon \Ob\cC \to 1\Mor\cC\]
    \item \emph{composition} operators\footnote{For simplicity of exposition, we'll allow abuses of notation when the context is clear. For instance, we may omit the composition symbol, or omit the subscript of composition operators, or shorten expressions such as $\alpha'\comp_0 \id_f$ to $\alpha'\comp_0f$ or $\alpha'f$.}
    \[\comp_0 \colon \Mor\cC \tttimes{\Ob\cC} \Mor\cC \to \Mor\cC,\quad\comp_0 \colon 2\Mor\cC \tttimes{\Ob\cC} \Mor\cC \to \Mor\cC,\]
    \[\text{and }\comp_1 \colon 2\Mor\cC \tttimes{\Mor\cC} 2\Mor\cC \to 2\Mor\cC\]
    defined for all pairs of ($2$-)morphisms $(f, g)$ such that the (dimension appropriate) source of the $g$ equals the target of $f$.
\end{itemize}

For instance, the large category $\cat$ of categories and functors can also be regarded as a $2$-category: the large $2$-category of categories, functors and natural transformations, with usual composition of functors and natural transformations.


\subsubsection{Equivalences in a $2$-category}

We can now define the notion of an equivalence in an arbitrary $2$-category. It is used e.g.~in \cite[\textsection 1.5]{StreetFibBicat}, \cite[\textsection 2]{lack1}, \cite[Definition~1.6]{GurskiBieq}, \cite[Definition~1.4.6]{RiehlVerityBook}, and it also recovers the notion of equivalence between categories when read in the case of the large $2$-category $\cC=\cat$.

\begin{defn}[Equivalence and adjoint equivalences in a $2$-category $\cC$]
\label{defequiin2cat}
A $1$-morphism $F\colon A\to B$ in a $2$-category $\cC$ is an \emph{equivalence}, and we write $F\colon A\simeq B$, if
there exists a $1$-morphism $G\colon B\to A$ in $\cC$ and isomorphisms
    \[\id_A\cong G\circ F\text{ in }\cC(A,A)\quad\text{ and }\quad F\circ G\cong \id_B\text{ in } \cC(B,B).\]
Two objects $A$ and $B$ in a $2$-category $\cC$ are \emph{equivalent}, and we write $A\simeq B$, if there exists an equivalence between them in $\cC$.
\end{defn}

It is easily verified that the relation $\simeq$ of being equivalent for categories or more generally for objects in a $2$-category $\cC$ is an equivalence relation.

There is a characterization of equivalences in a $2$-category in terms of \emph{adjoint equivalences}, a notion that we briefly recall. See e.g.~\cite[\textsection 1.5]{StreetFibBicat}, \cite[\textsection IV.4]{MacLane}, \cite[\textsection 3]{lack2}, and \cite[Theorem~1.9]{GurskiBieq} for more details.

\begin{defn}[Adjoint equivalence in a $2$-category $\cC$]
A $1$-morphism $F\colon A\to B$ in a $2$-category $\cC$ is an \emph{adjoint equivalence} in $\cC$ if
there exist a $1$-morphism $G\colon B\to A$ in $\cC$ and isomorphisms
    \[\eta\colon\id_A\cong G\circ F\text{ in }\cC(A,A)\quad\text{ and }\quad\varepsilon\colon F\circ G\cong \id_B\text{ in }\cC(B,B)\]
    and equalities
    \[\varepsilon F\circ F\eta=\id_F\text{ in }\cC(A,B)(F,F)\quad\text{ and }\quad G\varepsilon\circ\eta G=\id_G\text{ in }\cC(B,A)(G,G).\]
\end{defn}

The following result is classical. See
e.g.~\cite[Proposition I.12]{GabrielAbeliennes}, \cite[\textsection 1.5]{StreetFibBicat}, \cite[\textsection IV.4]{MacLane}, \cite[Lemma 5]{lack2}, and \cite[Theorem~1.9]{GurskiBieq} for a proof. Standard proofs of this fact use string calculus for $2$-morphisms in a $2$-category, clarifying the algebraic formulas which can be used instead.

\begin{thm}
\label{thmadjointeq}
A $1$-morphism $F\colon A\to B$ in a $2$-category $\cC$ is an equivalence if and only if it is an adjoint equivalence in $\cC$.
\end{thm}

The idea is that, by definition, every adjoint equivalence is an equivalence. Vice versa given an equivalence $f$ witnessed by the data of $g$, $\eta$ and $\varepsilon$, one could tweak $\varepsilon$ to obtain $\widetilde\varepsilon$ such that $g$, $\eta$ and $\widetilde\varepsilon$ witness that $f$ is an adjoint equivalence.

\subsubsection{The walking equivalence and adjoint equivalence}

We seek an indexing shape -- which we will denote by $\cE$ and features e.g.~in \cite[\textsection 3]{lack1} as $E$ -- parametrizing equivalences in a $2$-category. The construction of $\cE$ relies on the construction of the auxiliary category $\cP$, which in turns relies on the notion of suspension. Given a category $\cD$, we denote by $\Sigma\cD$ the \emph{suspension $2$-category}, which consists of two objects and one non-trivial hom-category given by $\cD$ (see e.g.~\cite[Definition 1.2]{NerveSuspension} for more details). This defines a functor $\Sigma\colon\cat\to\twocat$.

\begin{const}[The walking equivalence]
\label{constE}
We denote by $\cE$ the \emph{walking equivalence}, which is obtained from $\cP$ by freely adding
$2$-isomorphisms
    \[\id_a\cong g\circ f \quad\text{ and }\quad f\circ g\cong \id_b.\]
    The $2$-category $\cE$ can be understood as the pushout of $2$-categories
\[
\begin{tikzcd}[column sep=3.15cm]
\partial\cC_2\amalg\partial\cC_2\arrow[r,"{[(f\circ g, \id_b), (\id_a, g\circ f)]}"]\arrow[d, "\cong" swap]
\arrow[dddr, phantom, "\pushout", very near end, yshift=0.2cm, xshift=0.1cm]&\cP\arrow[ddd]\\
\Sigma\partial\cC_1\amalg\Sigma\partial\cC_1\arrow[d]&\\
\Sigma\cP\amalg\Sigma\cP\arrow[d]&\\
\Sigma\cI\amalg\Sigma\cI\arrow[r]&\cE
\end{tikzcd}
\]
Here, $\cC_2$ denotes the walking $2$-morphism, which consists of a $2$-morphism between two parallel $1$-morphisms, and the left vertical maps are suspensions of canonical maps that define $\cP$ and $\cI$ as pushouts in \cref{constP,constI}.
We see that $\Ob\cE\cong\{a, b\}$, but the hom-categories of $\cE$ can't be easily described.
\end{const}

By design, walking equivalence $\cE$ detects equivalences in the following sense.

\begin{rmk}
A $1$-morphism $F\colon A\to B$ in a $2$-category $\cC$ is an equivalence if and only if the functor $F\colon\cC_1\to\cC$ extends to a functor $\cE\to\cC$ along the inclusion $\cC_1\hookrightarrow \cE$. Meaning, if
there is a solution to the lifting problem of $2$-categories
\[\begin{tikzcd}
{\cC_1}\arrow[r,"F"]\arrow[d,hook, "f" swap]&\cC\\
\cE\arrow[ru,dashed]
\end{tikzcd}\]
\end{rmk}

By contrast, we can also determine the indexing shape -- which we will denote by $\cE^{\mathrm{adj}}$ and features e.g.~in \cite[\textsection 6]{lack2} as $E'$ -- parametrizing adjoint equivalences in a $2$-category.

\begin{const}[The walking adjoint equivalence]
\label{defEadj}
We denote by $\cE^{\mathrm{adj}}$ the \emph{walking adjoint equivalence}, which is obtained from $\cE$ by imposing the relations
\[\varepsilon f\circ f\eta=\id_f\quad\text{ and }\quad g\varepsilon\circ\eta g=\id_g.\] 
 The $2$-category $\cE^{\mathrm{adj}}$ can be understood as the pushout of $2$-categories
\[\begin{tikzcd}[column sep=3.15cm]
\partial \cC_3\amalg\partial\cC_3\arrow[r,"{[[\varepsilon f\circ f\eta, \id_f],[\id_g, g\varepsilon\circ\eta g]]}"]\arrow[d,"" swap]
\arrow[dr, phantom, "\pushout", very near end, yshift=0.2cm, xshift=0.1cm]&\cE\arrow[d]\\
\cC_2\amalg\cC_2\arrow[r]&\cE^{\mathrm{adj}}
\end{tikzcd}\]
Here $\partial\cC_3$ denotes the $2$-category consisting of two parallel $2$-morphisms the vertical map is a coproduct of two copies of the canonical \emph{folding} map $\partial\cC_3\to\cC_2$.
The set of objects of $\cE^{\mathrm{adj}}$ is $\Ob\cE^{\mathrm{adj}}:=\{a, b\}$ and the hom-categories
\[\cE^{\mathrm{adj}}(x,x)\cong\cE^{\mathrm{adj}}(x, y)\cong\cE^{\mathrm{adj}}( y, y)\cong\cE^{\mathrm{adj}}( y, x)\cong \chaos\mathbb N.\]
are all isomorphic to $\chaos\mathbb{N}$, the category with countably many objects and a unique morphism between any two objects.
\end{const}

By design, walking adjoint equivalence $\cE^{\mathrm{adj}}$ detects equivalences in the following sense.

\begin{rmk}
\label{Eadjcorepresents}
A $1$-morphism $F\colon A\to B$ in a $2$-category $\cC$ is an adjoint equivalence if and only if the functor $F\colon\cC_1\to\cC$ extends to a functor $\cE^{\mathrm{adj}}\to\cC$ along the inclusion $\cC_1\hookrightarrow \cE^{\mathrm{adj}}$. Meaning, if
there is a solution to the lifting problem of $2$-categories
\[\begin{tikzcd}
{\cC_1}\arrow[r,"F"]\arrow[d,hook,"f" swap]&\cC\\
\cE^{\mathrm{adj}}\arrow[ru,dashed]
\end{tikzcd}\]
\end{rmk}

Although we know by \cref{thmadjointeq} that equivalences and adjoint equivalences define the same notion, there is a sense in which the notion of adjoint equivalence is more \emph{coherent} than the basic equivalence. This is manifested at the level of indexing shapes $\cE$ and $\cE^{\mathrm{adj}}$: in $\cE$ the morphisms $\eta$ and $\varepsilon$ that have been added from $\cP$ are unrelated, and the extra coherence relation is then added in $\cE^{\mathrm{adj}}$. A precise formulation of this fact will be given in \cref{Enoncoh,Eadjcoh}, where we will discuss that $\cE^{\mathrm{adj}}$ is a \emph{contractible} $2$-category while $\cE$ is not. Note that, although perhaps counterintuitively, this is \emph{not} a contradiction to \cref{thmadjointeq}.

\begin{digression}
There is a notion of \emph{adjunction} in a $2$-category, meaning a generalization of the idea of an (adjoint) equivalence for which the $2$-morphisms $\eta$ and $\varepsilon$ are not necessarily invertible. This goes back to \cite[\textsection 2]{KellyAdjEnriched} and \cite[\textsection 2]{KellyStreet}. The corresponding corepresenting object, meaning what could be referred to as the \emph{walking adjunction}, is studied in \cite{SchanuelStreet} and also reused in \cite[Remark 3.3.8]{RiehlVerityMonads}. An alternative construction of $\cE^{\mathrm{adj}}$ from \cref{defEadj} is to impose the invertibility of $2$-morphisms corepresenting the unit and counit in the walking adjunction.
\end{digression}

\subsubsection{The fundamental category of a $2$-category}

The equivalence relation for objects in a $2$-category can also be tested as isomorphism relation in a suitable category or equality in a suitable set.

\begin{defn}
Given a $2$-category $\cC$, the \emph{fundamental category} of $\cC$ is the category \[\Pi_1\cC=((\Pi_0)_*)\cC\]
obtained by base change along the product-preserving functor $\Pi_0\colon\cat\to\set$. More explicitly,
its set of objects is
\[\Ob(\Pi_1\cC)=\Ob\cC\]
and its hom-sets are
\[(\Pi_1\cC)(A,B)=\Pi_0(\cC(A,B)).\]
This defines a product-preserving functor $\Pi_1\colon2\cat\to\cat$. 
\end{defn}

\begin{rmk}
\label{FundCatEqui2cat}
Given a $2$-category $\cC$ and $A,B\in\Ob\cC$,
the following are equivalent.
\begin{enumerate}[leftmargin=*]
\setcounter{enumi}{-1}
\item There is an equality
    \[A=B\text{ in the fundamental set }\Pi_0\Pi_1\cC.\] 
        \item There is an isomorphism
    \[A\cong B\text{ in the fundamental category }\Pi_1\cC.\]
    \item There is an equivalence
    \[A\simeq B\text{ in the $2$-category }\cC.\]
    \end{enumerate}
\end{rmk}

\subsection{Equivalence of $2$-categories and inside a $3$-category}

\subsubsection{Biequivalences of $2$-categories}
The notion of biequivalence of $2$-categories is classical and appears e.g.\ as \cite[\textsection 1.33]{StreetFibBicat} and \cite[\textsection 3]{lack1}.
\begin{defn}[Biequivalence of $2$-categories]
\label{defbiequi}
A $2$-functor $F\colon \cA\to \cB$ is a \emph{biequivalence of $2$-categories}, and we write  $F\colon \cA\simeq_2 \cB$, if 
\begin{enumerate}[leftmargin=*]
    \item the $2$-functor $F$ is \emph{surjective on objects up to equivalence}, meaning that for every object $b\in\Ob\cB$ there exists an object $a\in\Ob\cA$ and an equivalence 
    \[b\simeq Fa\text{ in the $2$-category } \cB;\]
    \item and the $2$-functor $F$ is a \emph{hom-wise equivalence},
    meaning that for all objects $a,a'\in\Ob \cA$
    the $2$-functor $F$ induces equivalences
    \[F_{a,a'}\colon \cA(a,a')\simeq \cB(Fa,Fa')\text{ of categories}.\]
    Two $2$-categories are \emph{biequivalent}, and we write $\cA\simeq_2 \cB$, if there exists a biequivalence between them.
\end{enumerate} 
\end{defn}

We immediately make use and explore the meaning of biequivalence of $2$-categories with the following two propositions, which feature in \cite{lack2} and are there attributed to Joyal. They formalize the idea that $\cE$ is not \emph{coherent}, while $\cE^{\mathrm{adj}}$ is.
\begin{prop}
\label{Enoncoh}
The $2$-category $\cE$ is not \emph{contractible}, namely
\[\cE\not\simeq_2\cC_0.\]
\end{prop}

We give two independent proofs of this fact.

\begin{proof}[Model categorical proof]
The lifting problem in $2\cat$
\[\begin{tikzcd}[column sep=3.15cm]
\partial \cC_3\arrow[d, "" swap]
\arrow[r,"{[\varepsilon f\circ f\eta,\id_f]}"]&\cE\\
\cC_2\arrow[ru,dashed, "?"]&,
\end{tikzcd}\]
where the left vertical map is the folding map,
does not admit a lift. Given that the left vertical map is a cofibration in the canonical model category $2\cat$ (see \cite{lack1, lack2} for an explicit description), the map $\cE\to\cC_0$ cannot be an acyclic fibration. Since it is a fibration, it means it cannot be a weak equivalence.
\end{proof}

\begin{proof}[Polygraphic homology proof]
The chain complex $\lambda\cE$ associated to the $2$-category $\cE$ as described in \cite[\textsection 4.2]{GuettaThesis} (based on ideas from \cite[\textsection 7]{BournTour})
is of the form
\[\lambda\cE=[\quad\mathbb Z[a,b]\xleftarrow{\partial_0}\mathbb Z[f,g]\xleftarrow{\partial_1}\mathbb Z[\eta,\varepsilon]\leftarrow 0\quad].\]
Given that
\begin{align*}
    \partial_1(\varepsilon)=-f-g\quad{ and }\quad\partial_1(\eta)=f+g,
\end{align*}
the second homology of $\lambda\cE$, which is the polygraphic homology of $\cE$, is given by
\[H_2(\lambda\cE)\cong\ker[\ \partial_1\colon\mathbb Z[\eta,\varepsilon]\to\mathbb Z[f,g]\ ]\cong\mathbb Z[\eta+\varepsilon]\cong\mathbb Z \not\cong0.\]
By \cite[Proposition 4.3.3]{GuettaThesis},
we obtain $\cE\not\simeq\cC_0$, as desired. 
\end{proof}

The following is discussed in \cite[\textsection 6]{lack2}.

\begin{prop}
\label{Eadjcoh}
The $2$-category $\cE^{\mathrm{adj}}$ is \emph{contractible}, namely
\[\cE^{\mathrm{adj}}\simeq_2\cC_0.\]
\end{prop}

\begin{proof}
Consider the unique functor $\varphi\colon\cE^{\mathrm{adj}}\to\cC_0$. There is an equivalence (in fact an equality)
\[\term\simeq \varphi(a)\text{ in }\cC_0\]
and $\varphi$ induces an equivalence 
\[\varphi_{a,a}\colon\cE^{\mathrm{adj}}(a,a)\cong\chaos\mathbb N\simeq\{\id_{\term}\}
=\cC_0(\term,\term)\text{ in }\cat.\]
Similarly, the functors $\varphi_{b,b}$, $\varphi_{a,b}$ and $\varphi_{b,a}$ are equivalences of categories. So there is a biequivalence $\varphi\colon\cE^{\mathrm{adj}}\simeq\cC_0$, as desired.
\end{proof}

\begin{rmk}
\label{cylinder2}
In the context of the canonical model structure on the category $2\cat$ (see \cite{lack1, lack2} for an explicit description), the object $\cE^{\mathrm{adj}}$ is a cylinder on the terminal object in the sense of \cite[Definition 1.2.4]{hovey}. Meaning, it comes with a factorization
\[\cC_0\amalg\cC_0\hookrightarrow\cE^{\mathrm{adj}}\stackrel\simeq\longrightarrow\cC_0\]
of the folding map of $\cC_0$ into a cofibration followed by a weak equivalence (in fact acyclic fibration). Note that $\cE^{\mathrm{adj}}$ is not the unique cylinder object, and shouldn't be expected to be. In fact, two other cylinder objects are given by $\cE^{\mathrm{adj}}\amalg_{\cC_0} \cE^{\mathrm{adj}}$ and $\cE^{\mathrm{adj}}\amalg_{\cC_1} \cE^{\mathrm{adj}}$, which are obtained by gluing two copies of $\cE^{\mathrm{adj}}$, respectively, identifying $a$ with $b$ or $f$ with $g$. Instead, the category $\cI$ -- regarded as a discrete $2$-category -- is \emph{not} a cylinder object in this context, as it is not cofibrant.
\end{rmk}

We now explore to which extent the notion of biequivalence is related to the existence of an inverse. One direction always holds:

\begin{rmk}
\label{biequicof}
A $2$-functor $F\colon \cA\to \cB$ between $2$-categories is a biequivalence if there exists a $2$-functor $G\colon \cB\to \cA$ and equivalences
    \[\eta\colon\id_\cA\simeq G\circ F\text{ in }2\cat(\cA,\cA)\text{ and }\varepsilon\colon F\circ G\simeq \id_\cB\text{ in }2\cat(\cB,\cB).\]
\end{rmk}

The hope for having equivalence of the two statements, in general, turns out to be false. We now recall how the other implication may fail.

\begin{rmk}
\label{ExampleLack}
If a $2$-functor $F\colon \cA\to \cB$ between $2$-categories is a biequivalence, it is not generally true that there exists a $2$-functor $G\colon \cB\to \cA$ and $2$-equivalences
    \[\eta\colon\id_\cA\simeq G\circ F\text{ in }2\cat(\cA,\cA)\text{ and }\varepsilon\colon F\circ G\simeq\id_\cB\text{ in }2\cat(\cB,\cB).\]
    It is shown in \cite[Example~3.1]{lack1} that a counterexample is given by a $2$-functor of the form $F\colon\cA\to B(\mathbb Z/2)$. Here, $\cA$ is the $2$-category whose underlying category is $B\mathbb{Z}$, with a $2$-morphism from the $1$-morphism $m$ to the $1$-morphism $n$ if and only if $m-n$ is even.
\end{rmk}

An alternative option is to consider an appropriate $2$-category $2\cat_{\mathrm{ps}}(\cA,\cB)$ of pseudo-functors from $\cA$ to $\cB$, for any $2$-categories $\cA$ and $\cB$. The following is mentioned e.g.~in \cite[\textsection 1.33]{StreetFibBicat} and in \cite[\textsection 3]{lack1}.

\begin{thm}
\label{biequivalenceweakinverse}
A $2$-functor $F\colon \cA\to \cB$ between $2$-categories is a \emph{biequivalence} if and only if there exists a pseudo-functor $G\colon \cB\to \cA$ and $2$-equivalences
    \[\eta\colon\id_\cA\simeq g\circ f\text{ in }2\cat_{\mathrm{ps}}(\cA,\cA)\text{ and }\varepsilon\colon f\circ g\simeq \id_\cB\text{ in }2\cat_{\mathrm{ps}}(\cB,\cB).\]
\end{thm}

\subsubsection{$3$-categories}

The viewpoint from \cref{biequicof,biequivalenceweakinverse}
is taken to define the notion of a biequivalence in any $3$-category. 

Recall that a \emph{$3$-category} $\cC$ consists of
\begin{itemize}[leftmargin=*]
    \item a \emph{set of objects} $\Ob\cC$ ,
    \item for every $c,c'\in\Ob\cC$ a \emph{hom-$2$-category} $\cC(c,c')$,
    \item for every $c\in\Ob\cC$, an \emph{identity $2$-functor} 
\[\id_c\colon\{\term\}\to\cC(c,c);\]
\item for every $c,c',c''\in\Ob\cC$, a \emph{composition $2$-functor} 
\[\circ\colon\cC(c,c')\times\cC(c',c'')\to\cC(c,c'');\]
\end{itemize}
satisfying same axioms as \eqref{associativity} and \eqref{unitality} read in $2$-categories, as opposed to sets.

In particular, the data of a $3$-category
$\cC$ determines a set $1\Mor\cC$ of \emph{$1$-morphisms}, a set $2\Mor\cC$ of \emph{$2$-morphisms} and a set $3\Mor\cC$ of \emph{$3$-morphisms} given by
\[1\Mor\cC\coloneqq\coprod_{c,c'\in\Ob\cC}\Ob(\cC(c,c')),\quad\quad2\Mor\cC\coloneqq\coprod_{c,c'\in\Ob\cC}1\Mor(\cC(c,c'))\]
\[\text{ and }\quad3\Mor\cC\coloneqq\coprod_{c,c'\in\Ob\cC}2\Mor(\cC(c,c')).\]
Amongst other structual operators, one also gets \emph{composition} operators $\comp_p \colon q\Mor\cC \times_{p\Mor\cC} q\Mor\cC \to q\Mor\cC$
    defined for all $q > p \geq 0$ and all pairs of $q$-cells $(f,g)$ such that $s_p(g) = t_p(f)$.
     
Given categories $\cC$ and $\cD$, a \emph{$3$-functor} $F\colon\cC\to\cD$ consists of
\begin{itemize}[leftmargin=*]
    \item a function on objects
\[F\colon\Ob\cC\to\Ob\cD\]
\item and a $2$-functor on hom-$2$-categories for every $c,c'\in\Ob\cC$
\[F_{c,c'}\colon\cC(c,c')\to\cD(Fc,Fc'),\]
\end{itemize}
satisfying the \emph{functorial properties} from \eqref{functorial1} and \eqref{functorial2}.

For instance, the large $2$-category $2\cat$ of $2$-categories, $2$-functors and $2$-natural transformations can also be regarded as a $3$-category: the (large) $3$-category of $2$-categories, $2$-functors, $2$-natural transformations, and modifications, with usual composition of functors and natural transformations and modifications. One could also consider 
$2\cat_{\mathrm{ps}}$: the large $3$-category of $2$-categories, pseudo-functors, pseudo-natural transformations, and modifications.

\subsubsection{Biequivalences and biadjoint biequivalences in a $3$-category}
The definition of biequivalence in a $3$-category appears e.g.~as \cite[Definition 3.5]{GPS}, and also plays a central role in \cite[\textsection2]{GurskiBieq}.

\begin{defn}[Biequivalence in a $3$-category $\cC$]
\label{DefGenBieq}
A $1$-morphism $F\colon A\to B$ in a $3$-category $\cC$ is a \emph{biequivalence}, and we write $F\colon A\simeq_2 B$,  if
there exists a $1$-morphism $G\colon B\to A$ in $\cC$ and equivalences
  \[\eta\colon\id_A\simeq G\circ F\text{ in }\cC(A,A)\text{ and }\varepsilon\colon F\circ G\simeq \id_B\text{ in }\cC(B,B).\]
Two objects are \emph{$2$-equivalent}, and we write $A\simeq_2 B$, if there exists an equivalence between them.
\end{defn}

\begin{rmk}
For a $2$-functor, the notion of biequivalence in $\cC$ specializes to that of biequivalence of $2$-categories when read in the case of the large $3$-category $\cC=2\cat_{\mathrm{ps}}$.
However, when read in the large $3$-category $2\cat$, it does \emph{not} specialize to \cref{DefGenBieq} (see \cref{ExampleLack}, cf.\ also the discussion of \cite[\textsection 7.5]{LackCompanion}).
\end{rmk}

It is easily verified that the relation $\simeq_2$ of being biequivalent for objects in a $3$-category $\cC$ is an equivalence relation.

There is a characterization of equivalences in a $3$-category in terms of biadjoint biequivalences, which we now recall. The following is \cite[Definition~2.3]{GurskiBieq} (see also \cite[Definition 3.12]{CCKS}).

\begin{defn}[Biadjoint biequivalence \`a la Gurski in a $3$-category $\cC$] 
A $1$-morphism $F\colon A\to B$ in a $3$-category $\cC$ is a \emph{biadjoint biequivalence} in $\cC$ if there exists a $1$-morphism $G\colon B\to A$ in $\cC$ and $2$-morphisms between $1$-morphisms
    \[\eta\colon\id_A\simeq G\circ F\text{ in }\cC(A,A)\text{ and }\varepsilon\colon F\circ G\simeq\id_B\text{ in }\cC(B,B).\]
    and $3$-morphisms
    \[\Phi\colon \varepsilon F\circ F\eta\cong\id_F\text{ in }\cC(A,B)(F,F)\quad\mbox{ and  }\quad \Psi\colon G\varepsilon\circ\eta G\cong\id_G\text{ in }\cC(B,A)(G,G)\]
    and satisfying the \emph{swallowtail relations}
    \[    \varepsilon \comp_1 (F \comp_0 \Psi) = \varepsilon \comp_1 (\Phi \comp_0 G)\quad\mbox{ and  }
           \quad (G \comp_0 \Phi) \comp_1 \eta = (\Psi \comp_0 F) \comp_1 \eta.\]
\end{defn}

The following is \cite[Theorem~4.5]{GurskiBieq}.

\begin{thm}
A $1$-morphism $F\colon A\to B$ in a $3$-category $\cC$ is a biequivalence if and only if it is a biadjoint biequivalence in $\cC$.
\end{thm}

\subsubsection{The walking biequivalence and bidajoint biequivalence}

We seek an indexing shape -- which we will denote by $\mathrm{bi}\cE$ -- parametrizing biequivalences in a $3$-category. Given a $2$-category $\cD$, we denote by $\Sigma\cD$ the \emph{suspension $3$-category}, which consists of two objects and one non-trivial hom-$2$-category given by $\cD$ (see e.g.~\cite[\textsection B.6.5]{AraMaltsiniotisJoin}). This defines a functor $\Sigma\colon2\cat\to3\cat$.

\begin{const}[The walking biequivalence]
We denote by $\mathrm{bi}\cE$
the \emph{walking biequivalence}, which is obtained from $\cP$ by freely adding the isomorphisms
  \[\Phi\colon \varepsilon f\circ f\eta\cong\id_f\text{ in }\cC(a,b)(f,f)\quad\text{ and  }\quad \Psi\colon g\varepsilon\circ\eta g\cong\id_g\text{ in }\cC(b,a)(g,g)\]
It can be described as the pushout of $3$-categories
\[
\begin{tikzcd}[column sep=3.15cm]
\partial \cC_2\amalg\partial \cC_2\arrow[r,"{[[g\circ f, \id_x], [\id_y, f\circ g]]}"]\arrow[d, "\cong" swap]
\arrow[dddr, phantom, "\pushout", very near end, yshift=0.2cm, xshift=0.1cm]&\cP\arrow[ddd]\\
\Sigma\partial\cC_1\amalg\Sigma\partial\cC_1\arrow[d]&\\
\Sigma\cP\amalg\Sigma\cP\arrow[d]&\\
\Sigma\cE\amalg\Sigma\cE\arrow[r]&\mathrm{bi}\cE
\end{tikzcd}
\]
Here, the left vertical maps are suspensions of canonical maps that define $\cP$ and $\cE$ as pushouts in \cref{constP,constE}.
\end{const}

By design, the walking biequivalence $\mathrm{bi}\cE$ detects biequivalences in the following sense.

\begin{rmk}
A $1$-morphism $F\colon A\to B$ in a $3$-category $\cC$ is a biequivalence if and only if there is a solution in to the lifting problem of $3$-categories
\[\begin{tikzcd}
{\cC_1}\arrow[r,"F"]\arrow[d,hook,"f" swap]&\cC\\
\mathrm{bi}\cE\arrow[ru,dashed]
\end{tikzcd}\]
\end{rmk}

We can also determine the indexing shape -- which we will denote by $\mathrm{bi}\cE^{\mathrm{adj}}$ --  parametrizing biadjoint biequivalences in a $3$-category.

\begin{const}[{Gurski's biadjoint biequivalence}]
We denote by $\mathrm{bi}\cE^{\mathrm{adj}}$
the walking biequivalence, which is obtained as the following iterated pushout
\[\begin{tikzcd}[column sep=3.15cm]
\partial \cC_2\amalg\partial \cC_2\arrow[r,"{[[g\circ f, \id_x] , [(\id_y, f\circ g]]}"]\arrow[d, "\cong" swap]
\arrow[ddr, phantom, "\pushout", very near end, yshift=0.2cm, xshift=0.1cm]&\cP\arrow[dd]\\
\Sigma\partial\cC_1\amalg\Sigma\partial\cC_1 \arrow[d]&&\\
\Sigma\cE\amalg\Sigma\cE\arrow[r]\arrow[d]
\arrow[dr, phantom, "\pushout", very near end, yshift=0.2cm, xshift=0.1cm]&\mathrm{bi}\cE\arrow[d]\\
\Sigma\cE^{\mathrm{adj}}\amalg\Sigma\cE^{\mathrm{adj}}\arrow[r]&\mathrm{bi}\cE'\arrow[dd]& \partial \cC_3 \amalg \partial \cC_3\arrow[l, "{[[\varepsilon f\circ f\eta, \id_f], [g\varepsilon\circ\eta g, \id_g]]}" swap]
\arrow[ddl, phantom, "{\mbox{\reflectbox{$\pushout$}}}", very near end, yshift=0.2cm, xshift=0.1cm]\arrow[d, "\cong"]\\
&&\Sigma^2\partial\cC_1\amalg\Sigma^2\partial\cC_1 \arrow[d]\\
\partial \cC_4\amalg\partial \cC_4\arrow[r, "{[[\varepsilon (f  \Psi) , \varepsilon  (\Phi  g)], [(g  \Phi)  \eta, (\Psi  f)  \eta]]}"]\arrow[d]
\arrow[dr, phantom, "\pushout", very near end, yshift=0.2cm, xshift=0.1cm]&\mathrm{bi}\cE'' \arrow[d]&\Sigma^2\cI\amalg \Sigma^2\cI \arrow[l]\\
\cC_3\amalg \cC_3\arrow[r]&\mathrm{bi}\cE^{\mathrm{adj}}&\\
\end{tikzcd}\]
Here, $\partial\cC_4$ denotes the $3$-category consisting of two parallel $3$-morphisms and the last vertical map is the sum of two copies of the canonical \emph{folding} map $\partial\cC_4\to\cC_3$ that identifies the two $3$-morphisms. Instead, the first two left (non-isomorphism) vertical maps are suspensions of canonical maps that define $\cE$ and $\cE^{\mathrm{adj}}$ as pushouts in \cref{constE,defEadj}, and the right (non-isomorphism) vertical map
is an iterated suspension of the canonical map $\partial\cC_1\hookrightarrow\cI$.
\end{const}

The $3$-category $\mathrm{bi}\cE^{\mathrm{adj}}$ is designed to detect biadjoint biequivalences:

\begin{rmk}
A $1$-morphism $F\colon A\to B$ in a $3$-category $\cC$ is a biequivalence if and only if there is a solution to the lifting problem of $3$-categories
\[\begin{tikzcd}
{\cC_1}\arrow[r,"F"]\arrow[d,hook, "f"']&\cC\\
\mathrm{bi}\cE^{\mathrm{adj}}\arrow[ru,dashed]
\end{tikzcd}\]
\end{rmk}

Like before, although biequivalences and biadjoint biequivalences define the same notion, there is a sense in which the notion of biadjoint biequivalence is more \emph{coherent} than the basic biequivalence.
A precise formulation of this fact will be given \cref{nEincoh,biEconj}, where we will discuss that $\mathrm{bi}\cE^{\mathrm{adj}}$ is conjecturally a \emph{contractible} $3$-category while $\mathrm{bi}\cE$ is not.

\begin{digression}
There several variants of notions of biadjunctions in a $3$-categories, meaning a generalization of the idea of a (biadjoint) biequivalence for which the $2$- and $3$-cells occurring are not necessarily invertible, even in a weak sense. These include \cite[Definition 2.1]{GurskiBieq}, \cite[Definition 6.2]{AraujoCoh3Cat}, \cite{LackPseudomonads}, \cite[Example 1.1.7]{VerityEnriched}, \cite[\textsection 6]{KellyElementary}, \cite[\textsection 3]{PstragowskiDualizable}.
\end{digression}

\subsection{Equivalences of $n$-categories and inside an $(n+1)$-category} 

After having explored the meaning of the appropriate notion of sameness for $n$-categories for the low values $n=0,1,2,3$, we are now ready to discuss this for arbitrary $n$. Unsurprisingly, setting this up requires an induction that builds on previous notions and terminology, so we fix the following conventions.
We follow the convention that
 \begin{itemize}[leftmargin=*]
     \item The relation $\simeq_{-1}$ is equality, the relation $\simeq_0$ is isomorphism, the relation $\simeq_1$ is equivalence and the relation $\simeq_2$ is biequivalence.
     \item $(-1)\cat$ is the set $\{\varnothing,\{\term\}\}$, $0\cat$ is the category $\set$, $1\cat$ is the category $\cat$.
     \item $\cE_{-1}$ is $\cC_0$, $\cE_0$ is $\cI$, $\cE_1$ is $\cE$ and $\cE_2$ is $\mathrm{bi}\cE$.
     \end{itemize}
 
With these conventions, and assuming to know the notions of $(n-1)$-category and $(n-1)$-functor, the category $(n-1)\cat$, the notion of $(n-1)$-equivalence of $(n-1)$-categories $\simeq_{n-1}$ and the notion of $(n-1)$equivalence in an $n$-category $\simeq_(n-1)$, we will recall by induction on $n>0$ the notions of $n$-category and $n$-functor, the category $n\cat$, the notion of $n$-equivalence of $n$-categories $\simeq_n$ and the notion of $n$-equivalence in an $(n+1)$-category $\simeq_n$.

Given the inductive nature of the definitions, making sense of \emph{something} (e.g.~\cref{nequiofncat}) anywhere in this section for a certain $n$, assumes one being able to make sense of \emph{everything} (e.g.~\cref{defnequicat}) in this section for $n-1$.
We include for the reader's convenience a table summarizing how some of the key definitions build on each other for different values of $n$:\\
\begin{center}
\ \\
\begin{tabular}{|ccl|ccl|}
\hline
\multicolumn{3}{|c|}{Notion} & \multicolumn{3}{c|}{Defined in}\\
\hline
iso of sets &=& $0$-eq of $0$-cats & Defn~\ref{defbijsets} &=& Defn~\ref{nequiofncat} for $n=0$\\
iso in a cat &=& $0$-eq in a $1$-cat & Defn~\ref{defisocat} &=& Defn~\ref{defnequicat} for $n=0$\\
eq of cats &=& $1$-eq of $1$-cats &  Defn~\ref{defequiofcat} &=& Defn~\ref{nequiofncat} for $n=1$\\
eq in a $2$-cat &=& $1$-eq in a $2$-cat &  Defn~\ref{defequiin2cat} &=& Defn~\ref{defnequicat} for $n=1$\\
bieq of $2$-cats &=&$2$-eq of $2$-cats &  Defn~\ref{defbiequi} &=& Defn~\ref{nequiofncat} for $n=2$\\
bieq in a $3$-cat&=&$2$-eq in a $3$-cat &  Defn~\ref{DefGenBieq}&=& Defn~\ref{defnequicat} for $n=2$\\
&&$3$-eq of $3$-cats &&& Defn~\ref{nequiofncat} for $n=3$\\
&&$3$-eq in a $4$-cat &&& Defn~\ref{defnequicat} for $n=3$\\
&&$4$-eq of $4$-cats &&& Defn~\ref{nequiofncat} for $n=4$\\
&&$4$-eq in a $5$-cat &&& Defn~\ref{defnequicat} for $n=4$\\
\multicolumn{3}{|c|}{$\dots$}&&&$\dots$\\
&&$n$-eq of $n$-cats & &&Defn~\ref{nequiofncat} for $n$ arbitrary\\
&&$n$-eq in an $(n+1)$-cat &&& Defn~\ref{defnequicat} for $n$ arbitrary\\
&&$\dots$&&&$\dots$\\
\hline
\end{tabular}
\end{center}

\subsubsection{$n$-categories}
Following the convention that a $0$-category
is a set (and recovering the fact that a $1$-category is a category), recall from e.g.~\cite[\textsection IV.2]{EilenbergKelly} (cf.~also \cite{EhresmannDouble}) that 
an \emph{$n$-category} $\cC$ consists of
\begin{itemize}[leftmargin=*]
    \item a \emph{set of objects} $\Ob\cC$,
    \item for every $c,c'\in\Ob\cC$ a \emph{hom-$(n-1)$-category} $\cC(c,c')$,
    \item for every $c\in\Ob\cC$, an \emph{identity $(n-1)$-functor} 
\[\id_c\colon\{\term\}\to\cC(c,c)\]
\item and, for every $c,c',c''\in\Ob\cC$,
a \emph{composition $(n-1)$-functor}
\[\circ\colon\cC(c,c')\times\cC(c',c'')\to\cC(c,c''),\]
\end{itemize}
satisfying same axioms as \eqref{associativity} and \eqref{unitality} read in $(n-1)$-categories as opposed to sets.

In particular, the data of an \emph{$n$-category}
$\cC$ also determines a collection of sets $q\Mor\cC$ for ${0<q\leq n}$, where by convention $0\Mor\cC=\Ob\cC$ is the set of objects of $\cC$ and $q\Mor\cC$
is the set of \emph{$q$-morphisms} given by
\[q\Mor\cC\coloneqq\coprod_{c,c'\in\Ob\cC} ((q-1)\Mor)(\cC(c,c')).\]
Amongst other structural operators, one also gets \emph{composition operators} $\comp_p \colon q\Mor\cC \times_{p\Mor\cC} q\Mor\cC \to q\Mor\cC$
    defined for all $q > p \geq 0$ and all pairs of $q$-cells $(f,g)$ such that $s_p(g) = t_p(f)$.
Given $n$-categories $\cC$ and $\cD$, an $n$-\emph{functor} $F\colon\cC\to\cD$ consists of
\begin{itemize}[leftmargin=*]
\item a function on objects
\[F\colon\Ob\cC\to\Ob\cD\]
\item 
and an $(n-1)$-functor on hom-$(n-1)$-categories for every $c,c'\in\Ob\cC$
\[F_{c,c'}\colon\cC(c,c')\to\cD(Fc,Fc'),\]
\end{itemize}
satisfying the \emph{functorial properties} from \eqref{functorial1} and \eqref{functorial2} read in $(n-1)$-categories as opposed to sets.

\subsubsection{$n$-equivalences of $n$-categories}

\label{equivalences of ncategories}

We now define the relation $\simeq_n$ of $n$-equivalence between $n$-categories inductively for $n\ge0$.
This approach is consistent with the one from \cite[\textsection 1]{StreetOrientedSimplexes} and is also the usual notion of equivalence in enriched contexts from e.g.~\cite[\textsection 1.11]{Kelly} or \cite[Def.\ 1.3.11]{JohnsonYauBook}.

\begin{defn}[$n$-equivalence of $n$-categories]
\label{nequiofncat}
An $n$-functor $F\colon \cA\to \cB$ between $n$-categories is an \emph{$n$-equivalence},  and we write $F\colon \cA\simeq_n \cB$, if
\begin{enumerate}[leftmargin=*]
    \item the $n$-functor $F$ is \emph{surjective on objects up to equivalence}, meaning that for every object $b\in\Ob\cB$ there exists an object $a\in\Ob\cA$ and an $(n-1)$-equivalence\footnote{Note that $(n-1)$-equivalence inside $\cB$ is defined either in previous sections (for $n\leq2$) or in \cref{defnequicat} (for $n\geq 2$).}
    \[b\simeq_{n-1} Fa\text{ in the $n$-category } \cB.\]
    \item the $n$-functor $F$ is a \emph{hom-wise $(n-1)$-equivalence}, meaning that for all objects $a,a'\in\Ob \cA$ the $n$-functor $F$ induces $(n-1)$-equivalences
    \[F_{a,a'}\colon \cA(a,a')\simeq_{n-1}\cB(Fa,Fa')\text{ of $(n-1)$-categories}.\]
\end{enumerate}
Two $n$-categories are \emph{$n$-equivalent}, and we write $\cA\simeq_n \cB$, if there is an $n$-equivalence between them.
\end{defn}

We now explore how the notion of $n$-equivalence guarantees the existence of a kind of inverse. Like before, one should not expect that an equivalence of $n$-categories should have an inverse (cf.~\cref{ExampleLack}), but one direction always holds.
The following is alluded to in \cite[\textsection 1]{GPS}, referring to \cite{StreetOrientedSimplexes}.

\begin{rmk}
An $n$-functor $F\colon \cA\to \cB$ between $n$-categories is an $n$-equivalence
if there exists an $n$-functor $G\colon \cB\to \cA$ and $(n-1)$-equivalences 
    \[\eta\colon\id_\cA\simeq_{n-1} G\circ F\mbox{ in } n\cat(\cA,\cA)\mbox{ and }\varepsilon\colon F\circ G\simeq_{n-1}\id_\cB\mbox{ in } n\cat(\cB,\cB).\]
\end{rmk}

The following notion of an $n$-equivalence in an $(n+1)$-category is treated in \cite[\textsection 1]{StreetOrientedSimplexes}.

\begin{defn}[$n$-equivalence in an $(n+1)$-category $\cC$]
\label{defnequicat}
An $1$-morphism $F\colon A\to B$ in an $(n+1)$-category $\cC$ is an \emph{$n$-equivalence} if
and only if there exist a $1$-morphism $G\colon B\to A$ and $(n-1)$-equivalences
    \[\eta\colon\id_A\simeq_{n-1} G\circ F\text{ in } \cC(A,A)\text{ and }\varepsilon\colon F\circ G\simeq_{n-1}\id_B\text{ in } \cC(B,B).\]
    Two objects of $\cC$ are \emph{$n$-equivalent}, and we write $A\simeq_n B$, if there is an $n$-equivalence between them.
\end{defn}

\subsubsection{The fundamental $k$-category of an $n$-category}

Testing whether two objects $A$ and $B$ in an $(n+1)$-category $\cC$ are $n$-equivalent can be tested as an $(n-1)$-equivalence in a suitable $n$-category $\Pi_n\cC$ instead, which we now describe.

\begin{defn}
Let $n>0$. Given an $(n+1)$-category $\cC$ the \emph{fundamental $n$-category} of $\cC$ is the $n$-category \[\Pi_{n}\cC=((\Pi_{n-1})_*)\cC\]
obtained by base change along the product-preserving functor $\Pi_{n-1}\colon n\cat\to(n-1)\cat$. This defines a product-preserving functor $\Pi_{n}\colon (n+1)\cat\to n\cat$. More explicitly,
its set of objects is
\[\Ob(\Pi_{n}\cC)=\Ob\cC\]
and its hom-$(n-1)$-categories are
\begin{equation}
\label{HomofPi}
    (\Pi_{n}\cC)(c,c')=\Pi_{n-1}(\cC(c,c')).
\end{equation}
For $0\leq k< n$, the \emph{fundamental $k$-category} of $\cC$ is by induction defined to be the $(k+1)$-category
\[\Pi_{k+1}\cC=(\Pi_{k})_*\cC\cong\Pi_{k+1}(\Pi_{k+2}(\dots((\Pi_{n}\cC)\dots))).\]
\end{defn}

The following is a consequence of the definitions.

\begin{prop}
\label{EquiTestTrunc}
Let $n\geq0$. Given an $(n+1)$-category $\cC$ and $A,B\in\Ob\cC$,
the following are equivalent.
\begin{enumerate}[leftmargin=*]
    \item There is an $n$-equivalence
    \[A\simeq_{n} B\text{ in the $(n+1)$-category }\cC.\]
    \item For some -- hence for all -- $0\leq k<n$ there is a $k$-equivalence
    \[A\simeq_{k} B\text{ in the fundamental $(k+1)$-category }\Pi_{k+1}\cC.\]
    \end{enumerate}
\end{prop}

\begin{proof}
We prove the statement by induction on $n=0$ further assuming $k=n-1$ for simplicity of exposition. The other values of $k$ can be treated similarly with a further induction on $k$.

The base of the induction, the case $n=1$, is \cref{FundCatEqui1cat}, and we now assume $n>1$. Let $A,B\in\Ob\cC$.
Having an $n$-equivalence
\[A\simeq_{n}B\text{ in the $(n+1)$-category }\cC\]
is equivalent to the existence of $F\in\Ob\cC(A,B)$ and $G\in\Ob\cC(B,A)$ so that there are $(n-1)$-equivalences
\[G\circ F\simeq_{n-1}\id_A\text{ in the $n$-category }\cC(A,A)\]
\[\text{and }F\circ G\simeq_{n-1}\id_{B}\text{ in the $n$-category }\cC(B,B).\]
This is in turn equivalent -- by induction hypothesis -- to the existence of $(n-2)$-equivalences
\[G\circ F\simeq_{n-2}\id_A\text{ in the $(n-1)$-category }\Pi_{n-1}(\cC(A,A))=(\Pi_{n}\cC)(A,A)\]
\[\text{and }F\circ G\simeq_{n-2}\id_{B}\text{ in the $(n-1)$-category }\Pi_{n-1}(\cC(B,B))=(\Pi_{n}\cC)(B,B).\]
Finally, this is equivalent to the existence of an $(n-1)$-equivalence
\[A\simeq_{n-1}B\text{ in the $n$-category }\Pi_{n}\cB.\]
This concludes the proof.
\end{proof}

This fact allows one to see that \cref{nequiofncat} is consistent with the approach to equivalences of weak $n$-categories from \cite{Tamsamani}, \cite[\textsection 20]{SimpsonBook}, \cite[\textsection\textsection 6.1-6.2]{PaoliBook}.

\subsubsection{The walking $n$-equivalence}

We seek an indexing shape -- which we will denote by $\eq{n}$ -- parametrizing $n$-equivalences in an $n$-category.  We follow the convention that $-1\cE=\cC_0$, we have $\cI$, $\cE$ and $\mathrm{bi}\cE$ are precisely $0\cE$, $1\cE$ and $2\cE$, and, assuming to know how $(n-1)\cE$ is defined we define $n\cE$ for general $n$.

Given an $n$-category $\cD$, we denote by $\Sigma\cD$ the \emph{suspension $(n+1)$-category}, which consists of two objects and one non-trivial hom-$n$-category given by $\cD$ (see e.g.~\cite[\textsection B.6.5]{AraMaltsiniotisJoin}) for more details. This defines a functor $\Sigma \colon n\cat\to(n+1)\cat$.

\begin{const}[The walking $n$-equivalence]
Let $n>0$. We denote by $n\cE$
the \emph{walking $n$-equivalence}, which is obtained from $\cP$ as the pushout of $(n+1)$-categories
\[
\begin{tikzcd}[column sep=3.15cm]
\partial\cC_2\amalg\partial\cC_2\arrow[r,"{[[\id_{a},g\circ f], [f\circ g, \id_{b}]]}"]
\arrow[ddr, phantom, "\pushout", very near end, yshift=0.2cm, xshift=0.1cm]\arrow[d, "\cong" swap]&\cP\arrow[dd]\\
\Sigma\partial\cC_1\amalg\Sigma\partial\cC_1\arrow[d]&\\
\Sigma\eq{(n-1)}\amalg\Sigma\eq{(n-1)}\arrow[r]&\eq{n}
\end{tikzcd}
\]
Here, the left (non-isomorphism) vertical map is the coproduct of two suspension of canonical maps that define $(n-1)\cE$ by induction hypothesis.
\end{const}

The $n$-category $\eq{n}$ is designed to detect $n$-equivalences:

\begin{rmk}
A $1$-morphism $F\colon A\to B$ in an $(n+1)$-category $\cC$ is an $n$-equivalence if and only if there is a solution to the lifting problem of $(n+1)$-categories
\[\begin{tikzcd}
{\cC_1}\arrow[r,"F"]\arrow[d,  "f" swap,hook ]&\cC\\
\eq{n}\arrow[ru,dashed]
\end{tikzcd}\]
\end{rmk}

The following could be proven with similar techniques to those from \cref{Enoncoh}.

\begin{prop}
\label{nEincoh}
Let $n>0$. The $(n+1)$-category $n\cE$ is not contractible, namely
\[n\cE\not\simeq_n\cC_0.\]
\end{prop}

We know for abstract reasons that for all $n$ there exists a contractible $(n+1)$ category $\widetilde{n\cE}$ with two objects. This can be obtained, for instance, by factoring the unique map $\cC_0\amalg\cC_0\to\cC_0$ as a cofibration followed by a weak equivalence
\[\cC_0\amalg\cC_0\hookrightarrow \widetilde{n\cE}\stackrel{\simeq}{\longrightarrow}\cC_0\]
in the canonical model structure on $n\cat$ from \cite[Theorem 6.1]{LMW}.

However, we are not aware of a way to construct the $(n+1)$-category $\widetilde{n\cE}$ explicitly beyond $n=1$ for which we have discussed $\cE^{\mathrm{adj}}$. Even for $n=2$, one can consider the candidate $\mathrm{bi}\cE^{\mathrm{adj}}$, which is likely to be contractible, but we are not aware of a proof.

\begin{conjecture}
\label{biEconj}
The $3$-category $\mathrm{bi}\cE^{\mathrm{adj}}$ is contractible, namely there is an equivalence of $3$-categories
\[\mathrm{bi}\cE^{\mathrm{adj}}\simeq_3\cC_0.\]
\end{conjecture}

\subsection{Equivalences of $\omega$-categories and inside an $\omega$-category}

\subsubsection{$\omega$-categories}

While the notion of an $n$-category foresees the presence of morphisms up to dimension $n$, allowing cells in all dimensions $k\ge0$ leads to the notion of an $\omega$-category. We refer the reader to e.g.~\cite[\textsection1]{StreetOrientedSimplexes}
 for a traditional approach to the definition of an $\omega$-category, but we briefly recall the main features here.

The data of an \emph{$\omega$-category}
$\cC$ consists of a collection of sets $q\Mor\cC$, ${q \geq 0}$,
where $0\Mor\cC$ is called the set of \emph{objects} of $\cC$ and $q\Mor\cC$, $q>0$,
is the set of \emph{$q$-cells} or \emph{$q$-arrows} or cells of \emph{dimension} $q$ of $\cC$, together with:
\begin{itemize}[leftmargin=*]
    \item \emph{source} and \emph{target} operators $s_p, t_p \colon q\Mor\cC \to p\Mor\cC$
    for all $q > p \geq 0$;
    \item \emph{identity} operators $\id_q \colon p\Mor\cC \to q\Mor\cC$ for all $q> p\ge0$; 
    \item \emph{composition} operators $\comp_p \colon q\Mor\cC \times_{p\Mor\cC} q\Mor\cC \to q\Mor\cC$
    defined for all $q > p \geq 0$ and all pairs of $q$-cells $(f,g)$ such that $s_p(g) = t_p(f)$.
\end{itemize}
Notice that this is equivalent to endowing every pair $(p\Mor\cC, q\Mor\cC)$, $q > p \geq 0$,
with the structure (but not the axioms, yet) of a category.
For all $r > q > p \geq 0$, we ask that the triple $(p\Mor\cC, q\Mor\cC, r\Mor\cC)$
together with all the relevant source, target, identity and composition operators
is a $2$-category.
In particular, for every $c,c'\in\Ob\cC$ there is a \emph{hom-$\omega$-category} $\cC(c,c')$.

An \emph{$\omega$-functor} $F \colon \cC \to \cC'$ between $\omega$-categories $\cC$ and $\cC'$ is a collection of maps $F_q \colon q\Mor\cC \to q\Mor\cC'$ for
$q \geq 0$ that preserves source, target, identity, and composition.
We denote by $\omegacat$ the (large) category of (small) $\omega$-categories and $\omega$-functors.

A cell in an $\omega$-category $\cC$ which is the identity of a lower dimensional cell is said to be \emph{trivial}. An $\omega$-category in which all $q$-cells are trivial for $q>n$ is precisely an $n$-category, and an $\omega$-functor between $n$-categories reduces precisely to an $n$-functor.

\subsubsection{$\omega$-equivalences of $\omega$-categories}

The following is from \cite[\textsection 1.2]{AraLucas} (see also \cite[D\'efinition~1.1.7]{LoubatonNerfs}).

\begin{defn}[Structure of reversibility]
Let $\cC$ be an $\omega$-category, and $S$ a set of morphisms of $\cC$. We say that $S$ is a \emph{structure of reversibility} if
for every $k\ge1$ and for every $k$-cell $\varphi\colon\alpha\to\beta$ in $S$ there exist in $S$ a $k$-cell $\gamma\colon\beta\to\alpha$ and $(k+1)$-cells \[\Phi\colon\id_{\alpha}\to \gamma \comp_{k-1} \varphi\quad\text{ and }\quad\Psi\colon \varphi\comp_{k-1}\gamma\to\id_{\beta}.\]
\end{defn}

The following terminology is consistent with \cite[Definition 4.2]{LMW} (and with \cite{PolyBook}).

\begin{defn}[$\omega$-equivalence in an $\omega$-category $\cC$]
A $1$-morphism $F\colon A\to B$ in an $\omega$-category $\cC$ is an \emph{$\omega$-equivalence}, and we write $F\colon A\simeq_\omega B$, if there exists a structure of reversibility in the $\omega$-category $\cC$ containing $F$.
Two objects $A$ and $B$ in $\cC$ are \emph{$\omega$-equivalent}, and we write $A\simeq_\omega B$, if there is an $\omega$-equivalence between them.
\end{defn}

It is easily verified that the relation $\simeq_\omega$ of being $\omega$-equivalent
for objects in an $\omega$-category $\cC$ is an equivalence relation.

The following recovers the notion of weak equivalence from \cite[Definition~4.7]{LMW} (and $\omega$-equivalence from \cite{PolyBook}).

\begin{defn}[$\omega$-equivalence of $\omega$-categories]
\label{equiofomegacat}
An $\omega$-functor $F\colon \cA\to \cB$ between $\omega$-categories is an \emph{$\omega$-equivalence}, and we write $F\colon \cA\simeq_\omega \cB$,
if and only if
\begin{enumerate}[leftmargin=*, ref=(\arabic*)]
    \item the $\omega$-functor $F$ is \emph{surjective on objects up to $\omega$-equivalence}, meaning that for every object $b\in\Ob\cB$ there exists an object $a\in\Ob\cA$ and an $\omega$-equivalence
    \[b\simeq_{\omega} Fa\text{ in } \cB\]
    \item\label{surjomegaequi} and the $\omega$-functor $F$ is \emph{surjective on $k$-morphisms up to $\omega$-equivalence}, meaning that, for all $k$-morphisms $x,x'$ in $\cA$ and all $(k+1)$-cells $\beta\colon F_kx \to F_kx'$ in $\cB$, there is a $(k+1)$-cell $\alpha\colon x \to x'$ in $\cA$ so that
    \[\beta\simeq_{\omega} F_{k+1}\alpha\quad\text{ in }\quad\cB(s_0(\beta),t_0(\beta))(\dots)(s_k(\beta),t_k(\beta)).\]
\end{enumerate}
Two $\omega$-categories $\cA$ and $\cB$ are \emph{$\omega$-equivalent}, and we write $\cA\simeq_\omega \cB$, if there is an $\omega$-equivalence between them.
\end{defn}

\begin{rmk}
In the context of \cref{equiofomegacat}, Condition \ref{surjomegaequi} can be replaced with
\begin{enumerate}[leftmargin=*]
    \item[(2')] the $\omega$-functor $F$ is a \emph{hom-wise $\omega$-equivalence}, meaning that for every $a,a'\in\Ob \cA$ the morphism $F$ induces an $\omega$-equivalence
    \[F_{a,a'}\colon \cA(a,a')\simeq_{\omega}\cB(Fa,Fa')\text{ of $\omega$-categories }.\]
\end{enumerate}
\end{rmk}

For an $\omega$-functor $F$, having an inverse implies being an equivalence of $\omega$-categories:

\begin{rmk}
An $\omega$-functor $F\colon \cA\to \cB$  is an $\omega$-equivalence if
there exists an $\omega$-functor $G\colon \cB\to \cA$ and $\omega$-equivalences 
    \[\eta\colon\id_\cA\simeq_{\omega} G\circ F\text{ in } \omega\cat(\cA,\cA)\text{ and }\varepsilon\colon F\circ G\simeq_{\omega}\id_\cB\text{ in } \omega\cat(\cB,\cB).\]
\end{rmk}

Once again, in general, one should not expect that an equivalence of $\omega$-categories should in general have an inverse (cf.~\cref{ExampleLack}).


\subsubsection{The walking $\omega$-equivalence}

We seek an indexing shape -- which we will denote by $\omega\cE$ -- parametrizing $\omega$-equivalences in an $\omega$-category. The same construction occurs in the literature as $R_1$ from \cite[Remark 4.4]{AraLucas}. The categorical structure of $\omega\cE$ essentially encodes the structure given by the set of \emph{witnesses} from
\cite{ChengOmegaDuals},
the \emph{reversibility} structure from \cite[Remark 4.4]{AraLucas} and
the \emph{quasi-invertible} structure from
\cite[\textsection 1.4]{RiceCoinductive}
\begin{const}
\label{mapnEtonext}
By induction on $n\geq0$, one gets a map
\[(n+1)\cE\to n\cE.\]
Indeed, for $n=0$ this is
the map
\[\cE\to\cI\]
induced at the pushout level by the map of spans
\[\begin{tikzcd}[column sep=3.15cm]
\Sigma\cI\amalg\Sigma\cI\arrow[d]&\partial\cC_2\amalg\partial\cC_2\arrow[l]\arrow[r,"{[[f\circ g, \id_b], [\id_a, g\circ f]]}"]\arrow[dd, equals]
&\cP\arrow[dd, equals]\\
\Sigma\cC_0\amalg\Sigma\cC_0 \arrow[d, "\cong" swap]&&\\
\cC_1\amalg\cC_1&\partial\cC_2\amalg\partial\cC_2\arrow[l]\arrow[r,"{[[f\circ g, \id_b], [\id_a, g\circ f]]}" swap]&\cP
\end{tikzcd}\]
and for $n>1$ it is the map
\[(n+1)\cE\to n\cE.\]
induced at the pushout level by the map of spans
\[\begin{tikzcd}[column sep=3cm]
\Sigma\eq n\amalg\Sigma\eq{n}\arrow[d]&\partial\cC_2\amalg\partial\cC_2\arrow[l]\arrow[r,"{[(f\circ g, \id_b), (\id_a, g\circ f)]}"]\arrow[d, equals]
&\cP\arrow[d, equals]\\
\Sigma\eq{(n-1)}\amalg\Sigma\eq{(n-1)}&\partial\cC_2\amalg\partial\cC_2\arrow[l]\arrow[r,"{[(f\circ g, \id_b), (\id_a, g\circ f)]}" swap]&\cP
\end{tikzcd}\]
\end{const}

\begin{const}[The walking $\omega$-equivalence]
We denote by $\omega\cE$ the \emph{walking $\omega$-equivalence}, obtained as the limit in $\omega$-categories of the maps from \cref{mapnEtonext}  \[\omega\cE:=\lim[\quad\dots\to\eq{(n+1)}\to \eq{n}\to\dots\to4\cE\to3\cE\to\mathrm{bi}\cE\to\cE\to\cI\quad].\]
\end{const}

The following explains how $\omega\cE$ relates with $n\cE$. Consider the intelligent truncation functor $\inttrunc{n}\colon\omega\cat\to n\cat$ from \cite[\textsection 1.2]{AraMaltsiniotisJoin}, which is the left adjoint to the canonical inclusion $n\cat\hookrightarrow \omega\cat$. Roughly, the functor $\inttrunc{n}$ universally enforces all morphisms in dimension $n+1$ to be equalities.

\begin{rmk}
If $\inttrunc{n}\colon\omega\cat\to n\cat$
denotes the intelligent truncation from \cite[\textsection 1.2]{AraMaltsiniotisJoin}, we have that for every $n$ there is an isomorphism of $n$-categories
\[\inttrunc{n}\omega\cE\cong n\cE.\]
\end{rmk}

The $\omega$-category $\eq{\omega}$ is designed to detect $\omega$-equivalences:

\begin{rmk}
A $1$-morphism $F\colon A\to B$ in an $\omega$-category $\cC$ is an $\omega$-equivalence if and only if there is a solution in $\omega\cat$ to the lifting problem
\[\begin{tikzcd}
{\cC_1}\arrow[r,"F"]\arrow[d,hook]&\cC\\
\eq{\omega}\arrow[ru,dashed]
\end{tikzcd}\]
\end{rmk}

The following can be proven using techniques similar to \cref{Enoncoh}.

\begin{prop}
Note that $\eq{\omega}$ is not contractible, namely
\[\eq{\omega}\not\simeq_\omega\cC_0.\]
\end{prop}

\subsubsection{The possibly coherent walking $\omega$-equivalence}

We know for abstract reasons (cf.~\cite[\textsection4.7]{LMW})
that there exists a contractible $\omega$-category $\widetilde{\omega\cE}$ with two objects, and it will be shown in the forthcoming manuscript \cite{PolyBook} that such $\omega$-category $\widetilde{\omega\cE}$ would automatically parametrize $\omega$-equivalences.

However, we are not aware of a known model for this $\omega$-category in the literature, and we consider a candidate $\widehat{\omega\cE}$ here. This is inspired by conversations with Rice related to \cite[Definition 11]{RiceCoinductive} and conversations with Ara, M\'etayer, and Mimram. It is also possibly related to  \cite[Construction 4.29]{HenryLoubaton}.

\begin{const}
\label{constructionJ}
We denote by $\cQ$ the free category generated by three morphisms $f\colon a\to b$, $g\colon  b\to a$ and $g'\colon b\to a$. This is obtained by gluing $f$ ``head-to-tail'' with both $g$ and $g'$,
and generating all possible compositions.
The set of objects is $\Ob\cQ=\{a, b\}$.
The category $\cQ$ can be understood as the pushout of categories
\[
\begin{tikzcd}[column sep=3.15cm]
\partial\cC_1\amalg\partial\cC_1^{\op}\amalg\partial\cC_1^{\op}\arrow[r,""]\arrow[d,hook] \arrow[dr, phantom, "\pushout", very near end, yshift=0.2cm, xshift=0.1cm]&\cC_0\amalg\cC_0\arrow[d]\\
\cC_1\amalg\cC_1^{\op}\amalg\cC_1^{\op}\arrow[r]&\cQ
\end{tikzcd}
\]
\end{const}

\begin{const}
\label{Jtruncated}
For $n\ge1$, we define inductively $\widehat{\omega\cE}^{(n)}$ to be an $\omega$-category (in fact an $n$-category) coming with a map $\widehat{\omega\cE}^{(n)}\to\widehat{\omega\cE}^{(n+1)}$:
\begin{itemize}[leftmargin=*]
\item Set $\widehat{\omega\cE}^{(0)}$ to be $\cC_0\amalg\cC_0$ and $\widehat{\omega\cE}^{(1)}$ to be $\cQ$.
\item For $n\ge 2$, set $\widehat{\omega\cE}^{(n)}$ to be the pushout of $\omega$-categories
\[
\begin{tikzcd}[column sep=3.15cm]
\Sigma(\widehat{\omega\cE}^{(n-2)})\amalg\Sigma(\widehat{\omega\cE}^{(n-2)}) \arrow[r, "{[[g\circ f,\id_a],[f\circ g', \id_b]]}"] \arrow[d] \arrow[dr, phantom, "\pushout", very near end, yshift=0.2cm, xshift=0.1cm]& \widehat{\omega\cE}^{(n-1)} \arrow[d] \\
\Sigma(\widehat{\omega\cE}^{(n-1)})\amalg\Sigma(\widehat{\omega\cE}^{(n-1)})\arrow[r]  &\widehat{\omega\cE}^{(n)} \\
\end{tikzcd}
\]
\end{itemize}
\end{const}

\begin{const}[The walking $\omega$-equivalence]
\label{constomegahat}
We denote by $\widehat{\omega\cE}$
the \emph{possibly coherent walking $\omega$-equivalence}, obtained as the colimit in $\omega$-categories
\[\widehat{\omega\cE}:=\colim[\quad\dots\leftarrow\widehat{\omega\cE}^{(n+1)}\leftarrow \widehat{\omega\cE}^{(n)}\leftarrow\dots\leftarrow\widehat{\omega\cE}^{(3)}\leftarrow\widehat{\omega\cE}^{(2)}\leftarrow\widehat{\omega\cE}^{(1)}\quad].\]
\end{const}

The $\omega$-category $\widehat{\omega\cE}$ is, evidently, an $\omega$-category with non-trivial morphisms in each dimension, and in particular it is not an $n$-category for any finite $n$. However, there are (at least) two natural ways to ``approximate'' it by an $n$-category, given by considering the left and right adjoint of the canonical inclusion $i_n\colon n\cat\to\omega\cat$, discussed in \cite[\textsection 1.2]{AraMaltsiniotisJoin}. We have already mentioned the left adjoint, the intelligent truncation $\inttrunc n\colon\omega\cat\to n\cat$, and we now consider as well the right adjoint, the \emph{rough truncation} $\roughtrunc n\colon\omega\cat\to n\cat$. Essentially, the functor $\roughtrunc n$ forgets all morphisms in dimension higher than $n$.

\begin{rmk}
By construction, the rough $n$-truncation of $\widehat{\omega\cE}$ is isomorphic to the $n$-th layer from \cref{Jtruncated}, namely
\[\roughtrunc n\widehat{\omega\cE}\cong\widehat{\omega\cE}^{(n)}.\]
Instead, the computation of the intelligent $n$-truncation of $\widehat{\omega\cE}$ is non-trivial. For lower levels, one can show that there are identifications
\begin{equation}
\label{inttruncofJ}  
\inttrunc{0}\widehat{\omega\cE}\cong\cC_0
\quad\text{ and }\quad
\inttrunc{1}\widehat{\omega\cE}\cong\cI
\quad\text{ and }\quad
\inttrunc{2}\widehat{\omega\cE}\simeq\cE^{\mathrm{adj}}.
\end{equation}
The computation for $\inttrunc{0}\widehat{\omega\cE}$ is a straightforward check. The one for $\inttrunc{1}\widehat{\omega\cE}$ relies on the fact that an isomorphism in a category can be described as discussed in \cref{IsoTwoInverses}. The one for $\inttrunc{2}\widehat{\omega\cE}$ is already delicate, and we briefly sketch an argument, leaving the details to the interested reader. If $c^{\natural}\colon\psh{t\Delta}_{(\infty,2)}\to2\cat$ denotes the left Quillen functor from \cite[Construction 4.8]{Nerves2Cat}, one can obtain
biequivalences 
\begin{align*}
    \inttrunc 2\widehat{\omega\cE}
     & \simeq c^{\natural}\left(\Delta[3]_{\mathrm{eq}} \aamalg{\Delta[1]_t\amalg\Delta[1]_t} (\Delta[0]\amalg \Delta[0])\right)\simeq c^{\natural}\Delta[0]\cong\cC_0\simeq\cE^{\mathrm{adj}}.
\end{align*}
\end{rmk}

We wonder whether $\widehat{\omega\cE}$ is contractible:

\begin{question}
\label{Q}
Is it true that the $\omega$-category $\widehat{\omega\cE}$ is \emph{contractible}, namely that there is an $\omega$-equivalence
\[\widehat{\omega\cE}\simeq_\omega\cC_0\text{ of $\omega$-categories}?\]
\end{question}

The potential contractibility of $\widehat{\omega\cE}$ can be formulated in several equivalent ways, as follows.

\begin{prop}
The following are equivalent.
\begin{enumerate}[leftmargin=*, ref=(\arabic*)]
    \item\label{HatOmega} There is an $\omega$-equivalence \[\widehat{\omega\cE}\simeq_\omega\cC_0 \text{ of $\omega$-categories}.\]
    \item\label{HatTrunc} For all $n\geq0$ there is an $n$-equivalence  \[\inttrunc{n}\widehat{\omega\cE}\simeq_{n}\cC_0\text{ of $n$-categories}.\]
        \end{enumerate}
\end{prop}

\begin{proof}
Recall from \cite[Theorem 4.39]{LMW} (resp.~\cite[Theorem 6.1]{LMW}) the canonical model structure on $\omega\cat$ (resp.~$n\cat$), in which every object is fibrant and the weak equivalences are precisely the equivalences of $\omega$-categories from \cref{equiofomegacat} (resp.~the $\omega$-equivalences of $n$-categories from \cref{nequiofncat}). The fact that \ref{HatOmega} implies \ref{HatTrunc} is a consequence of the fact that $\inttrunc{n}\colon\omega\cat\to n\cat$ is a left Quillen functor and $\widehat{\omega\cE}$ is cofibrant by construction, and we now show that \ref{HatTrunc} implies \ref{HatOmega}.

Saying that there is an equivalence of $\omega$-categories $\widehat{\omega\cE}\simeq_\omega\cC_0$ as in \ref{HatOmega} is equivalent to saying that the unique map $\widehat{\omega\cE}\to\cC_0$ is an acyclic fibration in $\omega\cat$. Given the explicit set of generating cofibrations for $\omega\cat$ from \cite[Theorem 4.39]{LMW}, the same statement amounts to solving for all $k\geq 0$ a generic lifting problem in $\omega\cat$ of the form \[\begin{tikzcd}
    \partial\cC_{k}\arrow[r]\arrow[d,hook]&\widehat{\omega\cE}\\
    \cC_{k}\arrow[ru, dashed]&
    \end{tikzcd}\]
    Given the naturality square (of the counit of the adjunction $i_{k+2}\dashv \roughtrunc{k+2}$ on the given map $\partial\cC_k\to\widehat{\omega\cE}$)
    \[\begin{tikzcd}
        \partial\cC_k\arrow[r]
        &\widehat{\omega\cE}\\
        \roughtrunc{k+2}\partial\cC_k\arrow[r]\arrow[u,"\cong"]&\roughtrunc{k+2}\widehat{\omega\cE}\arrow[u]
    \end{tikzcd}\]
we see that the map $\partial\cC_k\to\widehat{\omega\cE}$ must factor through $\roughtrunc{k+2}\widehat{\omega\cE}$ and it hence suffices to solve the lifting problem in $(k+2)\cat$
    \[\begin{tikzcd}
    \partial\cC_{k}\arrow[r]\arrow[d,hook]&\roughtrunc{k+2}\widehat{\omega\cE}\\
    \cC_{k}\arrow[ru, dashed]&
    \end{tikzcd}\]
    Given the isomorphism
    \[\roughtrunc{k+2}\widehat{\omega\cE}\cong\roughtrunc{k+2}\inttrunc{k+4}\widehat{\omega\cE}\]
it suffices to solve the lifting problem $(k+2)\cat$
    \[\begin{tikzcd}
    \partial\cC_{k}\arrow[r]\arrow[d,hook]&\roughtrunc{k+2}\inttrunc{k+4}\widehat{\omega\cE}\\
    \cC_{k}\arrow[ru, dashed]&
    \end{tikzcd}\]
By transposing along the adjunction $i_{k+2}\dashv\roughtrunc{k+2}$,
    it suffices to solve the lifting problem of $(k+4)$-categories
        \[\begin{tikzcd}
    \partial\cC_{k}\arrow[r]\arrow[d,hook]&\inttrunc{k+4}\widehat{\omega\cE}\\
    \cC_{k}\arrow[ru, dashed]&
    \end{tikzcd}\]
    Given that there is an equivalence $\inttrunc{k+4}\widehat{\omega\cE}\simeq_n\cC_0$ by \ref{HatTrunc} with $n=k+4$, and given that $\partial\cC_k\hookrightarrow\cC_k$ is a cofibration in $(k+4)\cat$, we obtain that
this lifting problem admits a solution, as desired.
\end{proof}

Given \eqref{inttruncofJ} -- combined with \cref{Icontractible,Eadjcoh} -- we know that the equivalent conditions \ref{HatOmega} and \ref{HatTrunc} hold for $n\leq2$, but we don't have a proof for higher $n$.

Indications towards the potential contractibility of $\widehat{\omega\cE}$ are the following.
\begin{itemize}[leftmargin=*]
    \item The polygraphic homology $H_*\widehat{\omega\cE}$ of the $\omega$-category $\widehat{\omega\cE}$ vanishes in positive degrees, and this is by \cite[\textsection 4.3]{GuettaThesis} a necessary condition for the contractibility of $\widehat{\omega\cE}$.
    \item It is plausible that the $\omega$-category $\widehat{\omega\cE}$ agrees with the construction from \cite[Definition 11]{RiceCoinductive}, which is shown to be contractible in the sense of \cite[Definition~21]{RiceCoinductive} in \cite[Theorem~22]{RiceCoinductive}, and this is likely a necessary condition for the contractibility of $\widehat{\omega\cE}$.
\end{itemize}

However, depending on how the construction from \cite[Construction 4.29]{HenryLoubaton} is related to \cref{constomegahat}, it is possible that \cite[Lemma 4.33]{HenryLoubaton} implies that the answer to \cref{Q} is no.

\section{Equivalences of and inside weak higher categories}

Throughout this section, we assume the basic language and theory of Kan complexes, which we refer to as \emph{spaces}, and of quasi-categories, which we refer to as \emph{$\infty$-categories}. In particular, we will allow ourselves to cite results about spaces and $\infty$-categories when needed.

Building on this prerequisite, we will then give an informal introduction to the notion of $(\infty,n)$-categories and the notion of appropriate equivalence between $(\infty,n)$-categories and inside an $(\infty,n+1)$-category. Although the details will be some times omitted, all the statement are rigorous.

Following this convention, an $\infty$-category will always refer to a quasi-category (so really a simplicial set that admits lifts of inner horns), while an $(\infty,1)$-category would refer to the general notion, of which quasi-categories are just one incarnation. To exemplify this convention, here's one possible sentence: \emph{An $(\infty,1)$-category can be presented by different models. For instance, it could be presented by an $\infty$-category, i.e., a quasi-category, or by a complete Segal space.}

We will continue reserving the use of calligraphic letters, such as $\cA, \cB$, $\cC$, $\cI$, for (strict!) $n$-categories, consistently with the previous section. Instead, we will reserve script letters, such as $\mathscr A, \mathscr B$, $\mathscr C$, $\mathscr I$ for $(\infty,n)$-categories or $\infty$-categories. To showcase the convention, here are some relevant examples that will feature in this section:
\begin{itemize}
\item $n\cat$ denotes the ($1$-)category of $n$-categories;
\item $n\mathscr Cat$ will denote the $\infty$-category of $n$-categories;
\item $(\infty,n)\mathscr Cat$
will denote the $\infty$-category of $(\infty,n)$-categories;
\item $\msset$ denotes the category of marked simplicial sets.
\item $\msset_{(\infty,n)}$ will denote the model structure for $(\infty,n)$-categories on the category of marked simplicial sets, whose underlying $\infty$-category is $(\infty,n)\mathscr Cat$;
\end{itemize}

\subsection{Review of $(\infty,n)$-categories}
\subsubsection{$(\infty,n)$-categories}

An \emph{$(\infty,0)$-category} is an $\infty$-groupoid, a.k.a.~a space. Spaces assemble into a (cartesian closed) $\infty$-category $(\infty,0)\mathscr Cat$. One way to construct it is as the $\infty$-category underlying the (cartesian) Kan--Quillen model structure for Kan complexes.

Let now $n>0$. Assuming to know what is an $(\infty,n-1)$-category and what is the cartesian $\infty$-category $(\infty, n-1)\mathscr Cat$ of $(\infty,n-1)$-categories, we now recall by induction on $n>0$ the idea of an $(\infty,n)$-category and the cartesian $\infty$-category $(\infty,n)\mathscr Cat$.

 Without loss of generality\footnote{Some features in this presentation of the notion are not intrinsic to the notion of an $(\infty,n)$-category, meaning that they are not invariant under the appropriate notion of equivalence that will be introduced in \cref{00equivalences}. One of these is, for instance, the set of objects. Although it may appear to be a potential issue, it is not.
 }
 we can assume that a (\emph{small}) \emph{$(\infty,n)$-category}
$\mathscr C$ for $n>0$ consists in particular of
\begin{itemize}[leftmargin=*]
\item a set $\Ob\mathscr C$ of objects
\item for every $c,c'\in\Ob\mathscr C$ a \emph{hom-$(\infty,n-1)$-category} $\mathscr C(c,c')$,
\item an identity operator function for every
$c\in\Ob\mathscr C$ 
\[\id_c\colon\{\term\}\to\mathscr C(c,c)\quad\quad\term \mapsto \id_c\]
\item 
and a \emph{composition} $1$-simplex for every $c,c',c''\in\Ob\mathscr C$ in the $\infty$-category $(\infty,n)\mathscr Cat$
\[\circ\colon\mathscr C(c,c')\times\mathscr C(c',c'')\to\mathscr C(c,c'')\quad\quad(f,g) \mapsto g\circ f\]
\end{itemize}
satisfying for every $c,c',c'',c'''\in\Ob\mathscr C$ the \emph{associativity axiom} from \eqref{associativity}, given by the commutativity of the diagram in the $\infty$-category $(\infty,n-1)\mathscr Cat$
\begin{equation}
\label{HoAssociativity}
\begin{tikzcd}[column sep=3cm]
\mathscr C(c,c')\times\mathscr C(c',c'')\times\mathscr C(c'',c''')\arrow[r,"{\id_{\mathscr C(c,c')}\times\circ}"]\arrow[d, "{\circ\times\id_{\mathscr C(c'',c''')}}" swap]&\mathscr C(c,c')\times\mathscr C(c',c''')\arrow[d,"\circ"]\\
\mathscr C(c,c'')\times\mathscr C(c'',c''')\arrow[r,"\circ" swap]
&\mathscr C(c,c''')
\end{tikzcd}
\end{equation}
for every $c,c'\in\Ob\mathscr C$ the \emph{unitality axiom} from \eqref{unitality}, also in the $\infty$-category $(\infty,n)\mathscr Cat$. A complete definition would also require coherent homotopies expressing the associativity and unitality of iterations of the composition operator. There are several ways to make this precise, including -- but not limited to -- categories strictly enriched over a cartesian closed model structure for $(\infty,n-1)$-categories as in \cite[Theorem 3.11]{br1}
or $\infty$-categories enriched over the $\infty$-categories of $(\infty,n-1)$-categories as in \cite[\textsection5,6]{GH}.

Given an $(\infty,n)$-category $\cC$, one sees -- by induction on $n$ -- that $\cC$ has $k$-morphisms for $k>0$ and several composition maps along morphisms of a lower dimension, similarly to the strict cases, where axioms are replaced with lots of coherence data. With respect to these composition maps, all morphisms in dimension higher than $n$ are invertible. This is consistent with the fact that
an $(\infty,1)$-category recovers the notion of an \emph{$\infty$-category}.

Assuming that an $(\infty,0)$-functor is a map of spaces, and assuming to know what is an $(\infty,n-1)$-functor, we now recall by induction on $n>0$ the idea of an $(\infty,n)$-category and the cartesian $\infty$-category $(\infty,n)\mathscr Cat$.
Given $(\infty,n)$-categories $\mathscr C$ and $\mathscr D$, for $n\ge0$, an $(\infty,n)$-\emph{functor} $F\colon\mathscr C\to\mathscr D$ consists in particular of
\begin{itemize}[leftmargin=*]
\item a function on objects
\[F\colon\Ob\mathscr C\to\Ob\mathscr D\]
\item 
and an $(\infty,n-1)$-functor on hom-$(\infty,n-1)$-categories for every $c,c'\in\Ob\mathscr C$
\[F_{c,c'}\colon\mathscr C(c,c')\to\mathscr D(Fc,Fc'),\]
\end{itemize}
satisfying the \emph{functorial properties} from \eqref{functorial1} and \eqref{functorial2} in the $\infty$-category $(\infty,n)\mathscr Cat$. Again, a complete definition would also require coherent homotopies expressing the associativity and unitality of iterations of the composition operator functorial properties involving more than two inputs.

$(\infty,n)$-categories assemble into a (cartesian closed) $\infty$-category $(\infty,n)\mathscr Cat$, of which the $0$-simplices are $(\infty,n)$-categories and the $1$-simplices are the $(\infty,n)$-functors.
One can build the (cartesian closed) $\infty$-category of $(\infty,n)$-categories $(\infty,n)\mathscr Cat$ using the formalism from \cite[Remark 5.7.13, \textsection 6.1]{GH} (see also \cite{HeineEnrichedInfty}), or by taking the underlying $\infty$-category of one of the (cartesian) model structures for $(\infty,n)$-categories.

Model structures
for $(\infty,1)$-categories include the Joyal model structure quasi-categories (from \cite[\textsection 6.1]{JoyalVolumeII},~\cite[Theorem 2.2.5.1]{htt}), the Bergner model structure for Kan-enriched categories (from \cite{bergner}), and the Lurie model structure for naturally marked quasi-categories (from \cite[Proposition 3.1.3.7]{htt}). Model structures for $(\infty,2)$-categories include the Lurie model structure for $\infty$-bicategories (from \cite[Theorem 4.2.7]{LurieGoodwillie} and \cite[Definition 6.1]{GHL}). Model structures for $(\infty,n)$-categories for general $n$ include the Verity model structure for saturated $n$-complicial sets (from \cite[Theorem~1.25]{or}), the Barwick model structure for $n$-fold complete Segal spaces (from \cite{BarwickThesis}), the Rezk model structure for complete Segal $\Theta_n$-spaces (from \cite[\textsection 11]{rezkTheta}), the Ara model structure for $n$-quasi-categories (from \cite[\textsection 5.17]{ara}), the Bergner--Lurie model structure for categories enriched over $(\infty,n-1)$-categories (from \cite[Theorem A.3.2.24]{htt} applied for instance to \cite[Example A.3.2.23]{htt}), and the Campion--Doherty--Kapulkin--Maehara model structure for saturated $n$-comical sets (from \cite{CKM,DKM}). To see that all these models are equivalent, and precisely that they have an underlying $\infty$-category equivalent to $(\infty,n)\mathscr Cat$, see -- amongst others -- \cite{Bergner3MS,br1,htt,LurieGoodwillie,ara,br2,DKM,GHL,LoubatonEqui}.

Historically, before the notion of $(\infty,n)$-category, other notions of weak $n$-categories were developed in order to encode phenomena of interest. These are sometimes referred to as \emph{weak $n$-categories} or \emph{$(n,n)$-categories}. The idea is that an $(\infty,n)$-category has $k$-morphisms for any $k\geq0$ and all axioms for $k$-morphisms encoding the categorical structure only hold up to an invertible $(k+1)$-morphism. By contrast, an $(n,n)$-category could be seen as a special case of $(\infty,n)$-category, for which there are no non-identity $k$-morphisms in dimension $k>n$.

In the lower dimensions ($n=2$, $n=3$, and to some extent $n=4$), fully algebraic descriptions of weak $n$-categories are available, for instance in the models of \emph{bicategories} \cite{Benabou}, \emph{tricategories} \cite{GPS} and \emph{tetracategories} \cite{TrimbleTetra}. For general $n$, there are also versions of the notion of $(n,n)$-category which are based on higher operads, with the original notion appearing in \cite{BataninWeakn}, and further variants in \cite[\textsection9-10]{LeinsterHigherOperads}. Other (non-fully-algebraic) models of $(n,n)$-categories for general $n$ are Tamsamani categories, originally from \cite{Tamsamani}  and further developed in \cite{HirschowitzSimpson,Pellissier,SimpsonBook,PaoliBook}.

Weak $n$-categories assemble into an $\infty$-category $(n,n)\mathscr Cat$.
For $n\geq0$ there is an $\infty$-category $(n,n)\mathscr Cat$ of \emph{weak $n$-categories} or $n$-truncated $(\infty,n)$-categories. A model for this $\infty$-category was originally constructed as a relative category in \cite{Tamsamani}, and was further studied in \cite{HirschowitzSimpson,Pellissier,SimpsonBook,PaoliBook}. Alternatively, this $\infty$-category can be constructed following \cite[Proposition~6.1.7]{GH} as an $\infty$-localization of the $\infty$-category $(\infty,n)\mathscr Cat$. See also \cite[Remark 6.1.3]{GH} for how the approaches are related.

\subsubsection{Some notable functors}

We discuss ways in which the theory of $(\infty,n)$-categories has to interact with others, such as the theory of strict $n$-categories from 
\cref{EquiStrict}
and the theory of $(\infty,k)$-categories for $k<n$.

\begin{rmk}
\label{InclusionHigher}
For $n\geq0$, there is an inclusion of $\infty$-categories
\[(\infty,n)\mathscr Cat\hookrightarrow
(\infty,n+1)\mathscr Cat,\]
morally given by the fact that any $(\infty,n)$-category can be regarded as an $(\infty,n+1)$-category with no non-identity morphisms in dimension higher than $n$.
This $\infty$-functor admits a right adjoint, called the \emph{core functor},
\[\core_{n}\colon(\infty,n+1)\mathscr Cat\to(\infty,n)\mathscr Cat,\]
which intuitively retains in dimension $n+1$ only the morphisms that are invertible.

The adjunction of $\infty$-categories 
\[(\infty,n)\mathscr Cat\rightleftarrows
(\infty,n+1)\mathscr Cat\colon\mathrm{core}_{n}\]
can be implemented in the model of $(\infty,n)$-categories given by Verity's model structure on saturated $n$-complicial sets. Precisely, it is the underlying adjunction of $\infty$-categories of the adjunction $\mathrm{th}_n\dashv\mathrm{sp}_n$ from \cite[Notation~13]{VerityComplicialI}, which can be shown to be a Quillen pair using \cite[Lemma 25, Corollary 108]{VerityComplicialI}.

In particular, for all $n>0$ there are $\infty$-colimit preserving inclusions of $\infty$-categories 
\[(\infty,0)\mathscr Cat\hookrightarrow(\infty,1)\mathscr Cat\hookrightarrow
(\infty,2)\mathscr Cat\to\dots\hookrightarrow
(\infty,n)\mathscr Cat.\]
Also, for $m\leq n$, there exists a functor
\begin{equation}
\label{totalcore}
\core_m\colon(\infty,n)\mathscr Cat\xrightarrow{\mathrm{core}_{n-1}}(\infty,n-1)\mathscr Cat\xrightarrow{}\dots\xrightarrow{}(\infty,m+1)\mathscr Cat\xrightarrow{\mathrm{core}_{m}} (\infty,m)\mathscr Cat.
\end{equation}
\end{rmk}

\begin{rmk}
For $n\geq0$, if $(\infty,n+1)\mathscr Cat_{*,*}$ denotes the $\infty$-category of bipointed $(\infty,n+1)$-categories, there is a \emph{hom functor}
\[\mathfrak Hom\colon(\infty,n+1)\mathscr Cat_{*,*}\to(\infty,n)\mathscr Cat,\]
which essentially extracts the hom at two given points.
The hom functor admits a left adjoint, called the \emph{suspension functor}
\[\mathfrak S \colon (\infty,n)\mathscr Cat\to(\infty,n+1)\mathscr Cat_{*,*},\]
which builds an $(\infty,n+1)$-category with exactly two objects and one interesting hom between the two objects given by the input.
The adjunction
\[\mathfrak S \colon (\infty,n)\mathscr Cat\rightleftarrows(\infty,n+1)\mathscr Cat_{*,*}\colon\mathfrak Hom\]
is treated as \cite[Definition\ 4.3.21]{GH} and
can be implemented in the model of $(\infty,n)$-categories given by Verity's model structure on saturated $n$-complicial sets via the Quillen pair $\Sigma\dashv\Hom$ from \cite[Lemma~2.7]{ORfundamentalpushouts}, in the model for complete Segal $\Theta_n$-spaces via the left Quillen functor $V[1]$ from \cite[Proposition~4.6]{rezkTheta}, or in the model of categories enriched over complete Segal $\Theta_{n-1}$-spaces.

In total, one can forget the base points and iterate the construction, obtaining a functor
\begin{equation}
\label{totalsuspension}
\mathfrak S^{k-1} \colon (\infty,1)\mathscr Cat\xrightarrow{\mathfrak{S}}\dots\xrightarrow{\mathfrak{S}}(\infty,n)\mathscr Cat\xrightarrow{\mathfrak{S}}(\infty,k)\mathscr Cat.
\end{equation}
\end{rmk}

The following records in which sense the core functor and the hom functor commute with each other. It can also be seen as a variant of \eqref{HomofPi} for $(\infty,n)$-categories.

\begin{prop}
\label{LemmaHom}
Let $n\geq0$.
Let $\mathscr C$ be an $(\infty,n+1)$-category and $A,B\in\Ob\mathscr C$. There is an $(n-1)$-equivalence
of $(\infty,n-1)$-categories
\[
\core_{n-1}(\mathscr C(A,B))\simeq_{n-1}(\core_{n}\mathscr C)(A,B)
\]
\end{prop}

\begin{proof}[Proof of \cref{LemmaHom}]
One can prove (for instance in the model of saturated $n$-complicial sets) that there is a commutative diagram of left adjoint $\infty$-functors:
\[\begin{tikzcd}
   (\infty,n-1)\mathscr Cat \arrow[r,"\mathfrak S"]\arrow[d,hook]&(\infty,n)\mathscr Cat_{*,*}\arrow[d,hook]\\
   (\infty,n)\mathscr Cat \arrow[r,"\mathfrak S" swap]&(\infty,n+1)\mathscr Cat_{*,*}
\end{tikzcd}
\]
So by \cite[Proposition 2.1.10]{RiehlVerityBook} one obtains a commutative diagram of the corresponding right adjoint $\infty$-functors:
\[\begin{tikzcd}
   (\infty,n-1)\mathscr Cat &(\infty,n)\mathscr Cat_{*,*}\arrow[l,"\mathfrak Hom" swap]\\
   (\infty,n)\mathscr Cat \arrow[u,"\core_{n-1}"]&(\infty,n+1)\mathscr Cat_{*,*}\arrow[l,"\mathfrak Hom"]\arrow[u,"\core_{n}" swap]
\end{tikzcd}\]
and this concludes the proof.
\end{proof}

Let $n\mathscr Cat$ denote the $\infty$-category of (strict) $n$-categories, which can be constructed as the underlying $\infty$-category of the canonical model structure for $n$-categories from \cite[Theorem 6.1]{LMW}. We briefly discuss the state of the art of the functors that relate the $\infty$-categories of strict and weak $n$-dimensional categories by distinguishing $n\leq2$ or the case of general $n$.

\begin{rmk}
\label{nerveinclusion}
For $n=0,1,2$, there is an inclusion
of $\infty$-categories\footnote{We say that $X\hookrightarrow Y$ is an \emph{inclusion of $\infty$-categories}
if it is a map of $\infty$-categories that is a hom-wise equivalence of Kan complexes and injective-on-objects up to equivalence in $Y$.}
\[n\mathscr Cat\hookrightarrow
(\infty,n)\mathscr Cat,\]
that essentially regards a strict $n$-category as an $(\infty,n)$-category in the natural way. This $\infty$-functor admits a left adjoint
\[\Pi_{n}\colon(\infty,n)\mathscr Cat\to n\mathscr Cat,\]
which intuitively enforces all morphisms of dimension higher than $n$ and coherence equivalences to be equalities.
The adjunction 
\begin{equation}
\label{AdjPi}
\Pi_{n}\colon(\infty,n)\mathscr Cat\rightleftarrows n\mathscr Cat
\end{equation}
can be realized as a nerve-categorification Quillen reflection in most models of $(\infty,n)$-categories presented by model categories.
For $n=0$, this can be easily implemented in Kan complexes, and for $n=1$, this can be done in quasi-categories, naturally marked quasi-categories, and complete Segal spaces (see e.g.~\cite[\textsection4.2]{MORNerves}).
For $n=2$, this was done in saturated $2$-complicial sets \cite{Nerves2Cat}, in $2$-quasi-categories \cite{CampbellHoCoherent}, in $2$-fold complete Segal spaces \cite{MoserNerve}, and in categories enriched over $(\infty,1)$-categories \cite[\textsection4.2]{MORNerves}, \cite[\textsection1]{GHL}.
For general $n$, the adjunction \eqref{AdjPi} is potentially realized by the Quillen pair from \cite[Definition 4.46, Theorem 4.50]{HenryLoubaton}.
\end{rmk}

Consider the functor
\[\Pi_0\mathrm{core}_0\colon (\infty,n)\mathscr Cat\xrightarrow{\mathrm{core}_0}(\infty,0)\mathscr Cat\xrightarrow{\Pi_0}0\cat=\set.\]
This functor does not preserve $(\infty)$-pullbacks, but the following lemma records a crucial weaker compatibility of this functor with pullbacks that will play a role in \cref{MClemma,MClemma2}.

\begin{lem}
\label{secretpropertyPi0}
Given functors of $(\infty,n)$-categories $\mathscr A\to\mathscr B$ and $\mathscr C\to\mathscr B$, we have that
\[\Pi_0\mathrm{core_0}\left(\mathscr A\ootimes{\mathscr B}\mathscr C\right)\neq\varnothing\quad\text{if and only if}\quad
\Pi_0\mathrm{core_0}\mathscr A\tttimes{\Pi_0\mathrm{core_0}\mathscr B}\Pi_0\mathrm{core_0}\mathscr C\neq\varnothing.\]
\end{lem}

\begin{proof}
Since $\mathrm{core}_0$ preserves $(\infty)$-limits, saying that
\[\Pi_0\mathrm{core_0}\left(\mathscr A\ootimes{\mathscr B}\mathscr C\right)\neq\varnothing\]
is equivalent to saying that
\[\Pi_0\left((\mathrm{core_0}\mathscr A)\ootimes{\mathrm{core_0}\mathscr B}(\mathrm{core_0}\mathscr C)\right)\neq\varnothing,\]
which can be seen (for instance using the argument from \cite{Pi0Stackexchange}) to be equivalent to saying that
\[\Pi_0\mathrm{core_0}\mathscr A\tttimes{\Pi_0\mathrm{core_0}\mathscr B}\Pi_0\mathrm{core_0}\mathscr C\neq\varnothing,\]
so we are done.
\end{proof}


\subsection{$(\infty,n)$-equivalence of $(\infty,n)$-categories}

\label{00equivalences}

We follow the convention that an \emph{$(\infty,-1)$-equivalence} inside an $(\infty,0)$-category is a path.

Let now $n>0$. Assuming to know what is an $(\infty,n-1)$-equivalence inside an $(\infty,n)$-category
we now define inductively for $n\geq0$ an $(\infty,n)$-equivalence inside an $(\infty,n+1)$-category.

\begin{defn}[$(\infty,n)$-equivalence in an $(\infty,n+1)$-category $\mathscr C$]
Let $n\geq0$. A $1$-morphism $F\colon A\to B$ in an $(\infty,n+1)$-category $\mathscr C$ is an \emph{$(\infty,n)$-equivalence} if
and only if there exist a $1$-morphism $G\colon B\to A$ and $(\infty,n-1)$-equivalences
    \[\eta\colon\id_A\simeq_{(\infty,n-1)} G\circ F\text{ in } \mathscr C(A,A)\text{ and }\varepsilon\colon F\circ G\simeq_{(\infty,n-1)}\id_B\text{ in } \mathscr C(B,B).\]
    Two objects in an $(\infty,n+1)$-category $A$ and $B$ are \emph{$(\infty,n)$-equivalent}, and we write $A\simeq_{(\infty,n)} B$, if there is an $(\infty,n)$-equivalence between them.
\end{defn}

We can give an immediate variant of the definition.

\begin{prop}
\label{twosideinverse}
Let $n\geq0$. A $1$-morphism $F\colon A\to B$ in an $(\infty,n+1)$-category $\mathscr C$ is an $(\infty,n)$-equivalence if
and only if there exist $1$-morphisms $G,G'\colon B\to A$ and $(\infty,n-1)$-equivalences
    \[\eta\colon\id_A\simeq_{(\infty,n-1)} G\circ F\text{ in } \mathscr C(A,A)\text{ and }\varepsilon'\colon F\circ G'\simeq_{(\infty,n-1)}\id_B\text{ in } \mathscr C(B,B).\]
\end{prop}

\begin{proof}
The forward implication is straightforward, by taking $G'=G$ and $\varepsilon'=\varepsilon$. For the backwards implication, we observe that there is an $(\infty,n-1)$-equivalence
\begin{align*}\Phi\colon G'&\simeq_{(\infty,n-1)}\id_{A}\circ G'\simeq_{(\infty, n-1)}(G\circ F)\circ G'\\
&\simeq_{(\infty,n-1)}G\circ (F\circ G') \simeq_{(\infty,n-1)} G\circ \id_B \simeq_{(\infty,n-1)} G
\end{align*}
One can then use $\Phi$ and $\varepsilon'$ to produce an appropriate $\varepsilon$ as
\[\varepsilon\coloneqq\varepsilon'\circ(F\circ \Phi^{-1}),\]
so we are done.
\end{proof}

\subsubsection{Alternative viewpoints}

We discuss ways that the notion of equivalence in an $(\infty,n+1)$-category has to interact with the notion of equivalence in a strict $(n+1)$-category from \cref{equivalences of ncategories} and in an $(\infty,k+1)$-category for $k<n$.

Recall the functor $\mathrm{core}_k\colon(\infty,n)\mathscr Cat\to(\infty,k)\mathscr Cat$ from \cref{InclusionHigher}. The following characterization of $(\infty,n)$-equivalences in terms of lower dimensional cores can be seen as a variant of \cref{EquiTestTrunc} for $(\infty,n)$-categories. It is similar to the approach originally taken in \cite{Tamsamani,HirschowitzSimpson,Pellissier,SimpsonBook,PaoliBook} to define equivalences in a weak $n$-category.

\begin{prop}
\label{statementwithcore}
Let $n>0$.
Let $\mathscr C$ be an $(\infty,n+1)$-category and $A,B\in\Ob\mathscr C$. The following are equivalent.
\begin{enumerate}[leftmargin=*]
    \item There is an $(\infty,n)$-equivalence
    \[A\simeq_{(\infty,n)} B\text{ in the $(\infty,n+1)$-category }\mathscr C.\]
    \item 
    For some -- hence for all -- $0\leq k+1< n+1$ there is an
    $(\infty,k)$-equivalence
    \[A\simeq_{(\infty,k)} B\text{ in the $(\infty,k+1)$-category }\mathrm{core}_{k+1}\mathscr C.\]
    \end{enumerate}
\end{prop}

We can now prove the proposition.

\begin{proof}
We prove the statement by induction on $n>0$ further assuming $k=n-1$ for simplicity of exposition. The other values of $k$ can be treated similarly with a further induction on $k$.
The base of the induction, the case $n+1=1$, $k+1=0$, can be shown using \cite[Proposition 1.14]{JoyalVolumeII}, and we now assume $n>1$. Let $A,B\in\Ob\mathscr C$.

Having an $(\infty,n)$-equivalence
\[A\simeq_{(\infty,n)}B\text{ in the $(\infty,n+1)$-category }\mathscr C\]
is equivalent to the existence of $F\in\Ob\mathscr C(A,B)$ and $G\in\Ob\mathscr C(B,A)$ so that there are $(\infty,n-1)$-equivalences
\[G\circ F\simeq_{(\infty,n-1)}\id_A\text{ in the $(\infty,n)$-category }\mathscr C(A,A)\quad\text{ and}\]
\[F\circ G\simeq_{(\infty,n-1)}\id_{B}\text{ in the $(\infty,n)$-category }\mathscr C(B,B).\]
This is in turn equivalent -- by induction hypothesis and \cref{LemmaHom}-- to the existence of $(\infty,n-2)$-equivalences
\[G\circ F\simeq_{n-2}\id_A\text{ in the $(\infty,n-1)$-category }\core_{n-1}(\mathscr C(A,A))\simeq_{n-1}(\core_{n}\mathscr C)(A,A)\]
\[\text{and }F\circ G\simeq_{n-2}\id_{B}\text{ in the $(\infty, n-1)$-category }\core_{n-1}(\mathscr C(B,B))\simeq_{n-1}(\core_{n}\mathscr C)(B,B).\]
Finally, this is equivalent to the existence of an $(\infty,n-1)$-equivalence
\[A\simeq_{(\infty,n-1)}B\text{ in the $(\infty,n)$-category }\core_{n}\mathscr C.\]
This concludes the proof.
\end{proof}

Recall the functor $\Pi_{n}\colon(\infty,n)\mathscr Cat\to n\mathscr Cat$ from \cref{nerveinclusion}. We discuss -- at least for low values of $n$ -- a characterization of $(\infty,n)$-equivalences in terms of $n$-equivalences after applying $\Pi_n$.

\begin{prop}
\label{statementwithPi}
Let $n+1=0,1,2$. Let $\mathscr C$ be an $(\infty,n+1)$-category, and $A,B\in\Ob\mathscr C$. The following are equivalent.
\begin{enumerate}[leftmargin=*]
    \item There is an $(\infty,n)$-equivalence
    \[A\simeq_{(\infty,n)} B\text{ in the $(\infty,n+1)$-category }\mathscr C.\]
    \item There is an $n$-equivalence
    \[A\simeq_{n} B\text{ in the fundamental $(n+1)$-category }\Pi_{n+1}\mathscr C.\]
    \end{enumerate}
\end{prop}

The case $n+1=0$ is essentially by definition of $\Pi_0$, the case $n+1=1$ is essentially done in \cite[\textsection 1.10]{joyalnotes}, and the case $n+1=2$ is addressed in \cite[Theorem 1.4.7]{RiehlVerityBook}. The statement possibly also holds for higher $n$, upon correct identification of the functor $\Pi_n$.

We follow the convention that
an \emph{$(\infty,0)$-equivalence between $(\infty,0)$-categories} is a homotopy equivalence. Let now $n>0$. Assuming to know 
what is an $(\infty,n-1)$-equivalence between $(\infty,n-1)$-categories,
we now define inductively for $n\geq0$
an $(\infty,n)$-equivalence between $(\infty,n)$-categories.

\begin{defn}[$(\infty,n)$-equivalence of $(\infty,n)$-categories]
\label{defeqinfty}
Let $n>0$. An $(\infty,n)$-functor $F\colon \mathscr A\to \mathscr B$ between $(\infty,n)$-categories is an $(\infty,n)$-equivalence if and only if
\begin{enumerate}[leftmargin=*]
    \item the $(\infty,n)$-functor $F$ is \emph{surjective on objects up to equivalence}, meaning that for every object
    $b\in\Ob\mathscr B$ there exists an object $a\in\Ob\mathscr A$ and an $(\infty,n-1)$-equivalence
    \[b\simeq_{(\infty,n-1)} Fa\text{ in the $(\infty,n)$-category }\mathscr B.\]
    \item the $(\infty,n)$-functor $F$ is a \emph{hom-wise $(\infty,n-1)$-equivalence}, meaning that for all objects $a,a'\in\Ob \mathscr A$ the morphism $F$ induces $(\infty,n-1)$-equivalences
    \[F_{a,a'}\colon \mathscr A(a,a')\simeq_{(\infty,n-1)}\mathscr B(Fa,Fa')\text{ of $(\infty,n-1)$-categories}.\]
\end{enumerate}
\end{defn}

We can prove that that every $(\infty,n)$-equivalence admits an inverse in a suitable sense.

\begin{prop}[$(\infty,n)$-equivalence of $(\infty,n)$-categories]
Let $n>0$. 
An $(\infty,n)$-functor $F\colon \mathscr A\to \mathscr B$ between $(\infty,n)$-categories is an $(\infty,n)$-equivalence if
and only if there exists an $(\infty,n)$-functor $G\colon \mathscr B\to \mathscr A$ and $(\infty,n-1)$-equivalences
    \[\eta\colon\id_\mathscr A\simeq_{(\infty,n-1)} G\circ F\text{ in the $(\infty,n)$-category } (\infty,n)\mathscr Cat(\mathscr A,\mathscr A)\]
    \[\text{ and }\varepsilon\colon F\circ G\simeq_{(\infty,n-1)}\id_\mathscr B\text{ in the $(\infty,n)$-category } (\infty,n)\mathscr Cat(\mathscr B,\mathscr B).\]
\end{prop}

The content of this proposition is essentially \cite[\textsection 5.5, 5.6]{GH}. Precisely, the backwards implication is \cite[Proposition 5.5.3]{GH}, and the forwards implication is roughly discussed in \cite[Remark 5.6.5]{GH}; a further closely related discussion appears in \cite[Corollary 5.6.3]{GH}. We give an alternative proof.

\begin{proof}
Without loss of generality, one can represent $F\colon\mathscr A\to\mathscr B$ as a map between fibrant objects in the Verity model structure $\msset_{(\infty,n)}$ for saturated $n$-complicial sets on marked simplicial sets from \cite[Theorem 1.25]{or}.
By \cref{defeqinfty}, saying that $F$ is an $(\infty,n)$-equivalence amounts to
being a hom-wise equivalence of $(\infty,n-1)$-categories and essentially surjective up to $(\infty,n-1)$-equivalence.
By \cite[Corollary 3.2.11]{Loubaton2} this is equivalent to saying that $F$ is a weak equivalence 
in $\msset_{(\infty,n)}$. 
Using \cite[Proposition 1.2.8]{hovey}, this is equivalent to saying that the map $F\colon\mathscr A\to\mathscr B$ is a homotopy equivalence in $\msset_{(\infty,n)}$, meaning that there exist a map $G\colon\mathscr B\to\mathscr A$ and homotopies in $\msset_{(\infty,n)}$
\[G\circ F\sim_{\msset_{(\infty,n)}}\id_\mathscr A\text{ and }F\circ G\sim_{\msset_{(\infty,n)}}\id_{\mathscr B}.\]
By definition of $(\infty,-1)$-equivalence, this is equivalent to saying that there are $(\infty,-1)$-equivalences
\[\eta\colon\id_\mathscr A\simeq_{(\infty,-1)} G\circ F\text{ in the $(\infty,0)$-category } \Map^{h}_{(\infty,n)\mathscr Cat}(\mathscr A,\mathscr A)\]
    \[\text{ and }\varepsilon\colon F\circ G\simeq_{(\infty,-1)}\id_\mathscr B\text{ in the $(\infty,0)$-category } \Map^{h}_{(\infty,n)\mathscr Cat}(\mathscr B,\mathscr B).\]
One can show, by direct verification, that the model structure $\msset_{(\infty,n)}$
for complicial sets is simplicial with mapping spaces given by
\[\Map_{(\infty,n)\mathscr Cat}(\mathscr A,\mathscr B)\cong\mathrm{sp}_0(\infty,n)\mathscr Cat(\mathscr A,\mathscr B).\]
So the previous statement is equivalent to saying that there are $(\infty,-1)$-equivalences
\[\eta\colon\id_\mathscr A\simeq_{(\infty,-1)} G\circ F\text{ in the $(\infty,0)$-category } (\mathrm{sp}_0((\infty,n)\mathscr Cat))(\mathscr A,\mathscr A).\]
    \[\text{ and }\varepsilon\colon F\circ G\simeq_{(\infty,-1)}\id_\mathscr B\text{ in the $(\infty,0)$-category } (\mathrm{sp}_0((\infty,n)\mathscr Cat))(\mathscr B,\mathscr B)\]
    By \cref{statementwithcore},
this is equivalent to saying that there are $(\infty,n-1)$-equivalences
    \[\eta\colon\id_\mathscr A\simeq_{(\infty,n-1)} G\circ F\text{ in the $(\infty,n)$-category } ((\infty,n)\mathscr Cat)(\mathscr A,\mathscr A)\]
    \[\text{ and }\varepsilon\colon F\circ G\simeq_{(\infty,n-1)}\id_\mathscr B\text{ in the $(\infty,n)$-category } ((\infty,n)\mathscr Cat)(\mathscr B,\mathscr B)\]
    as desired.
\end{proof}

\subsubsection{Walking equivalence}
We can determine an indexing shape parametrizing $(\infty,n)$-equivalences in an $(\infty,n+1)$-category.
To this end, let $\mathscr C_1$ and $\mathscr I$
denote the $(\infty,1)$-categories obtained by regarding the categories $\cC_1$ and $\cI$ as $(\infty,1)$-categories via the inclusion of $\infty$-categories $1\mathscr Cat\hookrightarrow(\infty,1)\mathscr Cat$ from \cref{nerveinclusion}.
They will also be regarded as $(\infty,n)$-categories for $n\geq1$ via the inclusion of $\infty$-categories $1\mathscr Cat\hookrightarrow(\infty,1)\mathscr Cat\hookrightarrow\dots\hookrightarrow (\infty,n)\mathscr Cat$. It is evident that
$\mathscr C_1$ classifies $1$-morphisms in an $(\infty,n+1)$-category, and we will now show that $\mathscr I$ detects $(\infty,n)$-equivalences in an $(\infty,n+1)$-category in a suitable sense.

\begin{prop}
\label{propeqinfty1}
A $1$-morphism $F\colon A\to B$ in an $(\infty,1)$-category $\mathscr C$ is an \emph{$(\infty,0)$-equivalence} if and only if there is a solution to the lifting problem in the $\infty$-category $(\infty,1)\mathscr Cat$
\[\begin{tikzcd}
\mathscr C_1\arrow[r,"F"]\arrow[d,  "f" swap,hook ]&\mathscr C\\
\mathscr I\arrow[ru,dashed, "\widetilde F" swap]
\end{tikzcd}\]
\end{prop}

To clarify the meaning of the proposition, we intend that there exists an $(\infty,1)$-functor $\widetilde F\colon\mathscr I\to\mathscr C$ and an $(\infty,-1)$-equivalence
\[\widetilde F\circ f\simeq_{(\infty,-1)}F\text{ in the $(\infty,0)$-category }(\infty,1)\mathscr Cat(\mathscr C_1,\mathscr C).\]

\begin{proof}
Without loss of generality, one can represent $\mathscr C$ as a quasi-category, and the map $F$ as a $1$-simplex $F\colon\Delta[1]\to\mathscr C$. Using \cite[Proposition~2.2]{DuggerSpivakMapping} (and the model of right hom space in a quasi-category from \cite[\textsection 1.2.2]{htt}), one can show that saying that $F$ is an $(\infty,0)$-equivalence is equivalent to saying that there is a solution to the lifting problem in the category $\sset$
\[\begin{tikzcd}
\Delta[1]\arrow[r,"F"]\arrow[d,  "f" swap,hook ]&\mathscr C\\
N\cI\arrow[ru,dashed, "\widetilde F" swap]
\end{tikzcd}\]
Using an argument similar to the one in the proof of \cref{MClemma}, one can further see that this is equivalent to saying that there exists a solution to the lifting problem
\[\begin{tikzcd}
\mathscr C_1\arrow[r,"F"]\arrow[d,  "f" swap,hook ]&\mathscr C\\
\mathscr I\arrow[ru,dashed, "\widetilde F" swap]
\end{tikzcd}\]
in the $\infty$-category $(\infty,1)\mathscr Cat$, as desired.
\end{proof}

We can now use the previous proposition about $(\infty,0)$-categories in an $(\infty,1)$-categories to prove the analog characterization for $(\infty,n)$-equivalence in an $(\infty,n+1)$-category.

\begin{prop}
\label{propeqinfty}
Let $n\geq0$. A $1$-morphism $F\colon A\to B$ in an $(\infty,n+1)$-category $\mathscr C$ is an \emph{$(\infty,n)$-equivalence} if and only if there is a solution to the lifting problem in the $\infty$-category $(\infty,n+1)\mathscr Cat$
\[\begin{tikzcd}
\mathscr C_1\arrow[r,"F"]\arrow[d,  "f" swap,hook ]&\mathscr C\\
\mathscr I\arrow[ru,dashed, "\widetilde F" swap]
\end{tikzcd}\]
\end{prop}

\begin{proof}
Saying that $F$ is an $(\infty,n+1)$-equivalence in the $(\infty,n+1)$-category $\mathscr C$ is by \cref{statementwithcore} equivalent to saying that $F$ is an $(\infty,0)$-equivalence in $\mathrm{core}_1\mathscr C$. By \cref{propeqinfty1}, this is equivalent to saying that there exists a solution in the $\infty$-category $(\infty,n+1)\mathscr Cat$ for
\[\begin{tikzcd}
\mathscr C_1\arrow[r,"F"]\arrow[d,  "f" swap,hook ]&\mathrm{core}_1\mathscr C\\
\mathscr I\arrow[ru,dashed, "\widetilde F" swap]
\end{tikzcd}\]
By transposing along the inclusion-core adjunction of $\infty$-functors from \eqref{totalcore} for $m=1$, this is equivalent to saying that there exists a solution in the $\infty$-category $(\infty,1)\mathscr Cat$ for
\[\begin{tikzcd}
\mathscr C_1\arrow[r,"F"]\arrow[d,  "f" swap,hook ]&\mathscr C\\
\mathscr I\arrow[ru,dashed, "\widetilde F" swap]
\end{tikzcd}\]
as desired.
\end{proof}

One could also study when a $k$-morphism of an $(n+1)$-category is an $(\infty,n+1-k)$-equivalence. To this end, recall the suspension functor $\mathfrak S^{k-1}\colon(\infty,1)\mathscr Cat\to(\infty,n)\mathscr Cat$ from \eqref{totalsuspension} and consider the $(\infty,k)$-categories $\mathfrak S^{k-1}\mathscr C_1$ and $\mathfrak S^{k-1}\mathscr I$.
They will also be regarded as $(\infty,n)$-categories for $n\geq k$ via the inclusion of $\infty$-categories $(\infty,k)\mathscr Cat\hookrightarrow\dots\hookrightarrow (\infty,n)\mathscr Cat$.
It can be seen, for instance using the model of complete Segal $\Theta_n$-spaces, that $\mathfrak S^{k-1}\mathscr C_1$ classifies $k$-morphisms in an $(\infty,n+1)$-category. We will now show that $\mathfrak S^{k-1}\mathscr I$ classifies $(\infty,k)$-equivalences in an $(\infty,n+1)$-category in a suitable sense.

\begin{defn}
Let $n\ge0$ and $k\ge1$. A $k$-morphism $F\colon A\to B$ in an $(\infty,n+1)$-category $\mathscr C$ is an \emph{$(\infty,n+1-k)$-equivalence} 
if and only if there is a solution to the lifting problem in the $\infty$-category $(\infty,n+1)\mathscr Cat$
\[\begin{tikzcd}
\mathfrak S^{k-1}\mathscr C_1\arrow[r,"F"]\arrow[d,  "f" swap,hook ]&\mathscr C\\
\mathfrak S^{k-1}\mathscr I\arrow[ru,dashed]
\end{tikzcd}\]
\end{defn}

\subsection{Model categorical techniques}

We recall from \cite{hirschhorn} a model-categorical \cref{MClemma} which will be a crucial tool in the sequel. This technical result will allow to interpret the notion of $(\infty,n)$-equivalence in an $(\infty,n+1)$-category presented by one of the usual models coming from model structures.

We denote by $\sim_{\cM}$ the (left) homotopy relation in $\cM$. We denote by $\bot_\cM$ (resp.~$\top_{\cM}$) the initial (resp.~terminal) object of $\cM$.

The following technical fact allows one to work with strict lifting problems as opposed to lifting problems up to homotopy (meaning, inside an $\infty$-category). We will use this proposition in \cref{inqcats,innmqcats,inoobicats,incomplicial}.

\begin{prop}
\label{MClemma}
Let $n\geq0$. Let $\cM$ be a 
model category for $(\infty,n+1)$-categories\footnote{in the sense of \cite[Definition 15.4]{BarwickSchommerPries}} 
Suppose that $\mathscr C_1$ and $\mathscr I$ are cofibrant objects 
and 
that we are given a factorization of $f$
\begin{equation}
    \label{lift}
    \begin{tikzcd}
    f \colon &[-1cm] \mathscr C_1 \arrow[r, hook, "\varphi"] & \widetilde{\mathscr I} \arrow[r, "\simeq" swap, "\psi"] & \mathscr I
    \end{tikzcd}
\end{equation}
as a cofibration $\varphi$ followed by a weak equivalence $\psi$ in $\cM$.
A $1$-morphism $F\colon \mathscr{C}_1\to \mathscr{C}$ in an $(\infty,n+1)$-category $\mathscr C$ is an $(\infty,n)$-equivalence  if and only if there is a solution to the (strict!) lifting problem in $\cM$
\[\begin{tikzcd}
{\mathscr C_1}\arrow[r,"F"]\arrow[d,  "\varphi" swap,hook ]&\mathscr C\\
{\widetilde{\mathscr{I}}}\arrow[ru,dashed]
\end{tikzcd}\]
\end{prop}

\begin{proof}
By definition, being an equivalence in the $(\infty,n+1)$-category $\mathscr{C}$ is equivalent to the existence of
a solution to the lifting problem in the $\infty$-category $(\infty,n+1)\mathscr Cat$
\[\begin{tikzcd}
\mathscr C_1\arrow[r,"F"]\arrow[d,  "f" swap,hook ]&\mathscr C\\
\mathscr I\arrow[ru,dashed, "\widetilde F" swap]
\end{tikzcd}\]
This is equivalent to saying that there exists an $(\infty,n+1)$-functor $\widetilde F\colon\mathscr I\to\mathscr C$ and an $(\infty,-1)$-equivalence
\[\widetilde F\circ f\simeq_{(\infty,-1)}F\text{ in the $(\infty,0)$-category }(\infty,n+1)\mathscr Cat(\mathscr C_1,\mathscr C).\]
Phrased in terms of mapping spaces, this is equivalent to saying
\[\pi_0(\infty,n+1)\mathscr Cat(\widetilde{\mathscr I},\mathscr C)\tttimes{\pi_0(\infty,n+1)\mathscr Cat({\mathscr C_1},\mathscr C)}*\neq\varnothing.\]
Since derived mapping spaces in the model category $\cM$ compute the hom-$(\infty,0)$-categories of the $\infty$-category $(\infty,n+1)\mathscr Cat$, this is equivalent to saying
\[\pi_0\Map^h_{\cM}(\widetilde{\mathscr I},\mathscr C)\tttimes{\pi_0\Map^h_{\cM}{(\mathscr C_1},\mathscr C)}*\neq\varnothing.\]
By \cite[Theorem 17.7.2]{hirschhorn}, this is equivalent to saying that
\[\cM(\widetilde{\mathscr I},\mathscr C)/_{\sim_{\cM}}\tttimes{\cM({\mathscr C_1},\mathscr C)/_{\sim_{\cM}}}*\neq\varnothing.\]
This in turn is equivalent to the existence of an up-to-homotopy lift
\[\begin{tikzcd}
\mathscr C_1\arrow[r,"F"]\arrow[d,  "\varphi" swap,hook ]\arrow[dr, phantom, "\simeq", very near start, yshift=0.0cm, xshift=0.2cm]&\mathscr C\\
\widetilde{\mathscr{I}} \arrow[ru,dashed, "\widetilde F" swap] &\phantom{x}
\end{tikzcd}\]
in $\cM$. 
By \cite[Corollary 7.3.12]{hirschhorn}, this is equivalent to having a strict lift in the same diagram, as desired. 
\end{proof}

With an analog proof, one can prove a more general statement, which recovers the previous one for $k=1$. Consider the map $\mathfrak S^{k-1}f\colon\mathfrak S^{k-1}\mathscr C_1\to\mathfrak S^{k-1}\mathscr I$.

\begin{prop}
\label{MClemma2}
Let $n\geq0$ and $k>0$.
Let $\cM$ be a 
model category for $(\infty,n+1)$-categories\footnote{in the sense of \cite[Definition 15.4]{BarwickSchommerPries}}.
Suppose that $\mathfrak S^{k-1}\mathscr C_1$ and $\mathfrak S^{k-1}\mathscr I$ are cofibrant objects.
Suppose also that
we are given a factorization of $\mathfrak S^{k-1}f$
\begin{equation}
    \label{lift2}
    \begin{tikzcd}
    \mathfrak S^{k-1}f \colon &[-1cm] \mathfrak S^{k-1}\mathscr C_1 \arrow[r, hook, "\varphi_k"] & \widetilde{\mathscr I}_k \arrow[r, "\simeq" swap, "\psi_k"] & \mathfrak S^{k-1}\mathscr I
    \end{tikzcd}
\end{equation}
as a cofibration $\varphi_k$ followed by a weak equivalence $\psi_k$ in $\cM$.
Then, a $k$-morphism $F\colon A\to B$ in an $(\infty,n+1)$-category $\mathscr C$ is an $(\infty,n+1-k)$-equivalence  if and only if there is a solution to the (strict!) lifting problem in $\cM$
\[\begin{tikzcd}
{\mathfrak S^{k-1}\mathscr C_1}\arrow[r,"F"]\arrow[d,  "\varphi_k" swap,hook ]&\mathscr C\\
{\widetilde{\mathscr I}_k}\arrow[ru,dashed]
\end{tikzcd}\]
\end{prop}

\begin{rmk}
\label{suspensionrmk}
Given an explicit implementation of $\mathfrak S^{k-1}$ as a homotopical and cofibration-preserving functor $\Sigma^{k-1}$, one can obtain an instance of \eqref{lift2} by applying $\Sigma^{k-1}$ to \eqref{lift}.
    \[
  \begin{tikzcd}
    \Sigma^{k-1}f \colon &[-1cm]\Sigma^{k-1}\mathscr C_1 \arrow[r, hook, "\Sigma^{k-1}\varphi"] &  \Sigma^{k-1}\widetilde{\mathscr I} \arrow[r, "\simeq" swap, "\Sigma^{k-1}\psi"] &  \Sigma^{k-1}\mathscr I.
    \end{tikzcd}
    \]
\end{rmk}

In order to apply \cref{MClemma}, there's interest in $\mathscr C_1$, $\mathscr I$ and in the most simple of their cofibrant replacements in the main model categories that model $(\infty,n+1)$-categories. In the next section we will interpret \cref{MClemma} (and sometimes \cref{MClemma2}) in the main model structures for $(\infty,n+1)$-categories.

\subsection{Equivalences in an $(\infty,1)$-category presented by a model}

We discuss the notion of $(\infty,0)$-equivalence in $(\infty,1)$-categories presented by quasi-categories, naturally marked quasi-categories, and Kan-enriched categories. The treatment of other models -- such as complete Segal space, saturated $1$-complicial sets, and $1$-comical sets --
are recovered for the case $n=1$ of \cref{ncomplsets,ncomical,cstheta}.

\subsubsection{Equivalences in a quasi-category}

Consider the Joyal model structure $\sset_{(\infty,1)}$ on simplicial sets for quasi-categories from \cite[\textsection 6.1]{JoyalVolumeII} or \cite[Theorem 2.2.5.1]{htt}. A model for the inclusion $1\mathscr Cat\hookrightarrow(\infty,1)\mathscr Cat$ from \cref{nerveinclusion} is implemented by the usual nerve
$N\colon\cat\to\sset_{(\infty,1)}$, which is a right Quillen functor. Hence, a model for the walking $1$-morphism $\mathscr C_1$ is the $1$-simplex $N\cC_1\cong\Delta[1]$, and a model for $\mathscr I$ is $N\cI$. A factorization of $f\colon\Delta[1]\to N\cI$ of the form \eqref{lift} is the trivial factorization
\[
\begin{tikzcd}
    f \colon &[-1cm] {\Delta[1]} \arrow[r, hook,"{f}"] & {N\cI} \arrow[r, "\simeq" swap, ""] & N\cI.
    \end{tikzcd}
\]

\begin{prop}
\label{inqcats}
A $1$-morphism $F\colon A\to B$ in a quasi-category $\mathscr C$ is an $(\infty,0)$-equivalence  if and only if there is a solution to either -- hence all -- of the following (strict!) lifting problems in $\sset$
\[\begin{tikzcd}
{\Delta[1]}\arrow[r,"F"]\arrow[d,  "f" swap,hook ]&\mathscr C\\
N\cI\arrow[ru,dashed]
\end{tikzcd}
\quad\text{ and }\quad
\begin{tikzcd}
{\Delta[1]}\arrow[r,"F"]\arrow[d,  "f" swap,hook ]&\mathscr C\\
\mathrm{sk}_2 N\cI\arrow[ru,dashed]
\end{tikzcd}
\quad\text{ and }\quad
\begin{tikzcd}
{\Delta[1]}\arrow[r,"F"]\arrow[d,  "{[1,2]}" swap,hook ]&\mathscr C\\
\Delta[0] \aamalg{\Delta[1]}\Delta[3]\aamalg{\Delta[1]}\Delta[0]\arrow[ru,dashed]
\end{tikzcd}
\]
\end{prop}

The first is an application of \cref{MClemma}, the second and third ones follow from \cite[Proposition 2.2]{DuggerSpivakMapping}.

\subsubsection{Equivalences in a naturally marked quasi-category}

Consider the Lurie model structure $\sset^+_{(\infty,1)}$ on marked simplicial sets for naturally marked quasi-categories from \cite[Proposition 3.1.3.7]{htt}.
A model for the inclusion $1\mathscr Cat\hookrightarrow(\infty,1)\mathscr Cat$ from \cref{nerveinclusion} is implemented by the naturally marked nerve $N^{\natural}\colon\cat\to\sset^+_{(\infty,1)}$ which marks the $1$-simplices that are isomorphisms (see \cite[\textsection 1]{GHL} for more details), which is a right Quillen functor. Hence, a model for the walking $1$-morphism $\mathscr C_1$ is $N^{\natural}\cC_1=\Delta[1]$, the minimally marked $1$-simplex, and a model for $\mathscr I$ is $N^{\natural}\cI=N^\sharp\cI$, the simplicial set $N\cI$ maximally marked.
Relevant factorizations of $f\colon\Delta[1]\to N^\sharp\cI$ of the form \eqref{lift} include the trivial factorization
\[
\begin{tikzcd}
    f \colon &[-1cm] {\Delta[1]} \arrow[r, hook,"{[0,1]}"] & {N^\sharp\cI} \arrow[r, "\simeq" swap, "{\id_{N^\sharp\cI}}"] & N^\sharp\cI,
    \end{tikzcd}
\]
and the factorization
\[
\begin{tikzcd}
    f \colon &[-1cm] {\Delta[1]} \arrow[r, hook,"{[0,1]}"] & {\Delta[1]_t} \arrow[r, "\simeq" swap, "f"] & N^\sharp\cI.
    \end{tikzcd}
\]

The following is an application of \cref{MClemma}.
\begin{prop}
\label{innmqcats}
A $1$-morphism $F\colon A\to B$ in a naturally marked quasi-category $\mathscr C$ is an $(\infty,0)$-equivalence  if and only if there is a solution to either -- hence all -- of the following (strict!) lifting problems in $\sset^+$
\[\begin{tikzcd}
{\Delta[1]}\arrow[r,"F"]\arrow[d,  "f" swap,hook ]&\mathscr C\\ N^\sharp\cI\arrow[ru,dashed]
\end{tikzcd}
\quad\text{ and }\quad
\begin{tikzcd}
{\Delta[1]}\arrow[r,"F"]\arrow[d, "{[0,1]}" swap,hook ]&\mathscr C\\
\Delta[1]_t\arrow[ru,dashed]
\end{tikzcd}
\]
\end{prop}

\subsubsection{Equivalences in a category enriched over Kan complexes}
Consider the Bergner model structure $\sset\cat_{(\infty,1)}$ on simplicial categories for Kan-enriched categories from \cite{bergner}.
A model for the inclusion $1\mathscr Cat\hookrightarrow(\infty,1)\mathscr Cat$ from \cref{nerveinclusion} is implemented by the base-change functor $\mathrm{disc}_*\colon\cat\to\sset\cat$ along the discrete inclusion $\mathrm{disc}\colon\set\to\sset$, which regards each category as a simplicial category with discrete hom-simplicial sets and which is a right Quillen functor. Hence, a model for the walking $1$-morphism $\mathscr C_1$ is $\mathrm{disc}_*\cC_1$. The following is an application of \cref{MClemma}.

\begin{prop}
\label{enrichedlemma}
A $1$-morphism $F\colon A\to B$ in a Kan-category $\mathscr C$ is an $(\infty,0)$-equivalence if and only if there is a solution of the following lifting problem in the category $\sset\cat$
\[\begin{tikzcd}
{\mathrm{disc}_*\cC_1}\arrow[r,"F"]\arrow[d,  "f" swap,hook ]&\mathscr C\\
\mathfrak C N\cI\arrow[ru,dashed] & \phantom{x}
\end{tikzcd}\]
\end{prop}

\subsection{Equivalences in an $(\infty,2)$-category presented by a model}

We discuss the notion of $(\infty,1)$-equivalence in $(\infty,2)$-categories presented by $\infty$-bicategories and quasi-categorically enriched categories. The treatment of other models -- such as $2$-complicial sets, saturated $2$-comical sets and complete Segal $\Theta_2$-spaces --
are recovered for the case $n=2$ of \cref{ncomplsets,ncomical,cstheta}.

\subsubsection{Equivalences in an $\infty$-bicategory}

Consider the Lurie model structure on scaled simplicial sets for $\infty$-bicategories from \cite[Theorem 4.2.7]{LurieGoodwillie} or \cite[Definition 6.1]{GHL}.
A model for the inclusion $1\mathscr Cat\hookrightarrow(\infty,1)\mathscr Cat\hookrightarrow(\infty,2)\mathscr Cat$ from \cref{nerveinclusion} is implemented by the functor $N^\sharp\colon\cat\to\sset^{\mathrm{sc}}$ which marks all $2$-simplices in the nerve of a category, and which is right Quillen. Hence, model for the walking $1$-morphism $\mathscr C_1$ is $N^{\sharp}\cC_1=\Delta[1]$, the $1$-simplex with the unique possible scaling, and a model for $\mathscr I$ is $N^\sharp\cI$, the simplicial set $N\cI$ with the maximal scaling. A factorization of $f\colon\Delta[1]\to N^\sharp\cI$ of the form \eqref{lift} is the trivial factorization
\[
\begin{tikzcd}
    f \colon &[-1cm] {\Delta[1]} \arrow[r, hook,"{f}"] & {N^\sharp\cI} \arrow[r, "\simeq" swap, ""] & N^\sharp\cI,
    \end{tikzcd}
\]

 The following is an application of \cref{MClemma}. See also \cite[Definition 1.25]{GHL} for other characterizations of $(\infty,1)$-equivalences in an $\infty$-bicategory.

\begin{prop}
\label{inoobicats}
A $1$-morphism $F\colon A\to B$ in an $\infty$-bicategory $\mathscr C$ is an $(\infty,1)$-equivalence  if and only if there is a solution to the following (strict!) lifting problem in $\sset^{sc}$
\[\begin{tikzcd}
{\Delta[1]}\arrow[r,"F"]\arrow[d,  "f" swap,hook ]&\mathscr C\\
N^\sharp\cI\arrow[ru,dashed]
\end{tikzcd}
\]
\end{prop}

\subsubsection{Equivalences in a category enriched over (naturally marked) quasi-categories}

Consider the Bergner--Lurie model structure $\sset_{(\infty,1)}\cat$ (resp.~$\sset^+_{(\infty,1)}\cat$) on simplicial categories (resp.~marked simplicial categories) for quasi-categorically enriched categories (resp.~categories enriched over naturally marked quasi-categories). The model structure $\sset_{(\infty,1)}\cat$ (resp.~$\sset^+_{(\infty,1)}\cat$) is an instance of 
 \cite[Theorem A.3.2.24]{htt} in the specific case of \cite[Example A.3.2.23]{htt} (resp.~\cite[Example A.3.2.22]{htt}).
A model for the inclusion $\mathscr Cat\hookrightarrow(\infty,1)\mathscr Cat\hookrightarrow(\infty,2)\mathscr Cat$ from \cref{nerveinclusion} is implemented by the right Quillen functor given by the base-change functor $\mathrm{disc}_*\colon\cat\to\sset^{(+)}_{(\infty,1)}\cat$ along the discrete inclusion $\mathrm{disc}\colon\set\to\sset$.
The following are applications of \cref{MClemma}.

\begin{prop}
\label{enrichedlemma2}
A $1$-morphism $F\colon A\to B$ in a category $\mathscr C$ enriched over naturally marked quasi-categories is an $(\infty,1)$-equivalence if and only if there is a solution to the following lifting problem in the category $\sset^+\cat$:
\[
\begin{tikzcd}
{\mathrm{disc}_*\cC_1}\arrow[r,"F"]\arrow[d,  "f" swap,hook ]&\mathscr C\\
\mathfrak C^{sc}N^{\sharp}\cI\arrow[ru,dashed]&\phantom{x}
\end{tikzcd}
\]
\end{prop}

\begin{prop}
\label{enrichedlemma2}
A $1$-morphism $F\colon A\to B$ in a quasi-categorically-enriched category $\mathscr C$ is an $(\infty,1)$-equivalence if and only if there is a solution to the following lifting problem in $\sset_{(\infty,1)}\cat$:
\[\begin{tikzcd}
{\mathrm{disc}_*\cC_1}\arrow[r,"F"]\arrow[d,  "f" swap,hook ]
&\mathscr C\\
N_*\cE^{\mathrm{adj}}\arrow[ru,dashed]&\phantom{x}
\end{tikzcd}
\]
\end{prop}

\begin{ex}
\label{counterexample}
We want to stress that having a cofibration $\varphi$ in the factorization ot type \eqref{lift} is indeed crucial. 
Consider the $1$-morphism $f\colon a\to b$ in the $2$-category $\cE$. We know by construction that $f$ is an equivalence in the $2$-category $\cE$. In particular, it is also an $(\infty,1)$-equivalence in the quasi-categorically enriched category $N_*\cE$. We observe that the lifting problem up to homotopy admits a solution 
in the model category $\sset_{(\infty,1)}\cat$. 
\[\begin{tikzcd}
{N_*\cC_1}\arrow[r,"F"]\arrow[d,  "f" swap ]\arrow[dr, phantom, "\simeq", very near start, yshift=0.0cm, xshift=0.2cm]&N_*\cE\\
N_*\cI\arrow[ru,dashed] &\phantom{x}
\end{tikzcd}\]
The same lifting problem regarded strictly in the category $\sset\cat$ -- as opposed to up-to-homotopy in the model category $\sset_{(\infty,1)}\cat$ -- does \emph{not} lift. This can be seen by observing that all maps $N_*\cI\to N_*\cE$
are constant.
However, if one considers the canonical map $N_*\cE^{\mathrm{adj}}\to N_*\cI$, the lifting problem
\[\begin{tikzcd}
{N_*\cC_1}\arrow[r,"F"]\arrow[d,  "f" swap,hook ]&N_*\cE\\
N_*\cE^{\mathrm{adj}}\arrow[ru,dashed]
\end{tikzcd}\]
admits a strict solution in $\sset\cat$.
\end{ex}

\subsection{Equivalences in an $(\infty,n)$-category presented by a model}

We discuss the notion of $(\infty,n)$-equivalence in $(\infty,n+1)$-categories presented by saturated $(n+1)$-complicial sets, saturated $(n+1)$-comical sets and complete Segal $\Theta_{n+1}$-spaces.

\subsubsection{Equivalence in a saturated $(n+1)$-complicial set} 
\label{ncomplsets}

Consider the Verity model structure on marked simplicial sets $\msset_{(\infty,n+1)}$ for saturated $(n+1)$-complicial sets from \cite[Theorem~1.25]{or}.
A model for the inclusion $1\mathscr Cat\hookrightarrow(\infty,1)\mathscr Cat\hookrightarrow(\infty,n+1)\mathscr Cat$ from \cref{nerveinclusion} is implemented by the homotopical functor $N^\natural\colon\cat\to\msset_{(\infty,n+1)}$, which marks the $1$-simplices that are witnessed by an isomorphism and all simplices in dimension $2$ or higher in the nerve of a category. Hence, a model for the walking $1$-morphism $\mathscr C_1$ is $\Delta[1]$, the standard $1$-simplex minimally marked and a model for $\mathscr I$ is $N^\sharp\cI$, the simplicial set $N\cI$ maximally marked. Relevant factorizations of $f\colon\Delta[1]\to N^\sharp\cI$ of the form \eqref{lift}
include the factorization
\[
\begin{tikzcd}
    f \colon &[-1cm] {\Delta[1]} \arrow[r, hook,"{}"] & {N^\sharp\cI} \arrow[r, "\simeq" swap, "{\psi}"] & N^\sharp\cI,
    \end{tikzcd}
\]
and the factorization
\[
\begin{tikzcd}
    f \colon &[-1cm] {\Delta[1]} \arrow[r, hook,"{}"] & {\Delta[1]_t} \arrow[r, "\simeq" swap, "{\psi}"] & N^\sharp\cI,
    \end{tikzcd}
\]
and the factorization 
\[
\begin{tikzcd}
    f \colon &[-1cm] {\Delta[1]} \arrow[r, hook,"{}"] & {\Delta[3]_{\mathrm{eq}}} \arrow[r, "\simeq" swap, "{\psi}"] & N^\sharp\cI,
    \end{tikzcd}
\]

The following is an application of \cref{MClemma}.
\begin{prop}
\label{incomplicial}
A $1$-morphism $F\colon A\to B$ in a saturated $(n+1)$-complicial set $\mathscr C$ is an $(\infty,n-1)$-equivalence if and only if there is a solution to either -- hence all -- of the following (strict!) lifting problems in the category $\msset$
\[\begin{tikzcd}
{\Delta[1]}\arrow[r,"F"]\arrow[d,  "f" swap,hook ]&\mathscr C\\
N\cI^\sharp\arrow[ru,dashed]
\end{tikzcd}
\quad\text{and}\quad
\begin{tikzcd}
{\Delta[1]}\arrow[r,"F"]\arrow[d,  "f" swap,hook ]&\mathscr C\\
\Delta[1]_t\arrow[ru,dashed]
\end{tikzcd}
\quad\text{and}\quad
\begin{tikzcd}
{\Delta[1]}\arrow[r,"F"]\arrow[d,  "f" swap,hook ]&\mathscr C\\
\Delta[3]_{\mathrm{eq}}\arrow[ru,dashed]
\end{tikzcd}
\]
\end{prop}

Let $k>0$. A model for the suspension functor $\mathfrak S^{k-1}\colon(\infty,1)\mathscr Cat\to(\infty,k)\mathscr Cat$ is given by the homotopical functor $\Sigma^{k-1}\colon\msset_{(\infty,1)}\to\msset_{(\infty,k)}$ obtained by iterating the construction from \cite[Lemma~2.7]{ORfundamentalpushouts}. Hence, a model for $\mathfrak S^{k-1}\mathscr C_1$ is $\Sigma^{k-1}\Delta[1]$, and 
a model for $\mathfrak S^{k-1}\mathscr I$ is $\Sigma^{k-1}N^\sharp\cI$.
A factorization of $f_k\colon\Sigma^{k-1}\Delta[1]\to \Sigma^{k-1}N^\sharp\cI$ of the form \eqref{lift}
is the trivial factorization
\[
\begin{tikzcd}
    f_k \colon &[-1cm] {\Sigma^{k-1}\Delta[1]} \arrow[r, hook,"{}"] & {\Sigma^{k-1}N^\sharp\cI} \arrow[r, "\simeq" swap, "{\psi}"] & \Sigma^{k-1}N^\sharp\cI.
    \end{tikzcd}
\]

The following then is again an application of \cref{MClemma2}.
See also \cite[Corollary 3.2.11]{Loubaton2} for other characterizations of $(\infty,n+1-k)$-equivalences.

\begin{prop}
A $k$-morphism $F\colon A\to B$ in a saturated $(n+1)$-complicial set $\mathscr C$ is an $(\infty,n+1-k)$-equivalence if and only if there is a solution to the following (strict!) lifting problem in the category $\msset$
\[\begin{tikzcd}
{\Sigma^{k-1}\Delta[1]}\arrow[r,"F"]\arrow[d,  "f" swap,hook ]&\mathscr C\\
\Sigma^{k-1}N^\sharp\cI\arrow[ru,dashed]
\end{tikzcd}
\]
\end{prop}

\subsubsection{Equivalence in a saturated $(n+1)$-comical set}

\label{ncomical}

Consider the Doherty--Kapulkin--Maehara model structure on marked cubical sets for saturated $(n+1)$-comical sets
\cite[Theorem 2.7]{DKM}.
A model for the inclusion $1\mathscr Cat\hookrightarrow(\infty,1)\mathscr Cat\hookrightarrow(\infty,n+1)\mathscr Cat$ from \cref{nerveinclusion} is implemented by the homotopical functor $N^\natural\colon\cat\to mc\set_{(\infty,n+1)}$, which marks the $1$-cubes that are witnessed by an isomorphism and all cubes in dimension $2$ or higher in the cubical nerve of a category. Hence, a model for the walking $1$-morphism $\mathscr C_1$ is $\square[1]$, the standard $1$-cube minimally marked and a model for $\mathscr I$ is $N^\sharp\cI$, the cubical nerve $N\cI$ of $\cI$ maximally marked. Relevant factorizations of $f\colon\square[1]\to N^\sharp\cI$ of the form \eqref{lift}
include the factorization
\[
\begin{tikzcd}
    f \colon &[-1cm] {\square[1]} \arrow[r, hook,"{}"] & {N^\sharp\cI} \arrow[r, "\simeq" swap, "{\psi}"] & N^\sharp\cI,
    \end{tikzcd}
\]
and the factorization
\[
\begin{tikzcd}
    f \colon &[-1cm] {\square[1]} \arrow[r, hook,"{}"] & {\widetilde\square[1]} \arrow[r, "\simeq" swap, "{\psi}"] & N^\sharp\cI,
    \end{tikzcd}
\]
where $\widetilde\square[1]$ is the marked $1$-cube from \cite[\textsection1]{DKM},
and the factorization
\[
\begin{tikzcd}
    f \colon &[-1cm] {\square[1]} \arrow[r, hook,"{}"] & {L_{i,j}} \arrow[r, "\simeq" swap, "{\psi}"] & N^\sharp\cI,
    \end{tikzcd}
\]
where $L_{i,j}$ is the cubical set from \cite[\textsection2]{DKM}.

The following is an application of \cref{MClemma2}.
\begin{prop}
A $1$-morphism $F\colon A\to B$ in a saturated $(n+1)$-comical set $\mathscr C$ is an $(\infty,n)$-equivalence if and only if there is a solution to either -- hence all -- of the following (strict!) lifting problems in the category $mc\set$
\[\begin{tikzcd}
{\square[1]}\arrow[r,"F"]\arrow[d,  "f" swap,hook ]&\mathscr C\\
N^\sharp\cI\arrow[ru,dashed]
\end{tikzcd}
\quad\text{and}\quad
\begin{tikzcd}
{\square[1]}\arrow[r,"F"]\arrow[d,  "f" swap,hook ]&\mathscr C\\
{\widetilde\square[1]}\arrow[ru,dashed]
\end{tikzcd}
\quad\text{and}\quad
\begin{tikzcd}
{\square[1]}\arrow[r,"F"]\arrow[d,  "f" swap,hook ]&\mathscr C\\
L_{i,j}\arrow[ru,dashed]
\end{tikzcd}
\]
\end{prop}

\subsubsection{Equivalence in a complete Segal $\Theta_{n+1}$-space}
\label{cstheta}

Consider the Rezk model structure on $\Theta_{n+1}$-spaces for complete Segal $\Theta_{n+1}$-spaces from \cite[\textsection 11]{rezkTheta}.

A model for the inclusion $1\mathscr Cat\hookrightarrow(\infty,1)\mathscr Cat\hookrightarrow(\infty,n+1)\mathscr Cat$ from \cref{nerveinclusion} is implemented by the homotopical functor \[p^*N^{\textrm{Rezk}}\colon\cat\xrightarrow{N^{\textrm{Rezk}}}\spsh{\Theta_1}_{(\infty,1)}\xrightarrow{p^*}\spsh{\Theta_{n+1}}_{(\infty,n+1)},\]
where $N^{\textrm{Rezk}}$ denotes the Rezk nerve from \cite[\textsection 3.5]{rezk} -- there referred to as the \emph{classifying diagram} -- which is a right Quillen functor, and hence homotopical, and $p\colon\Theta_{n+1}\to\Theta_1=\Delta$ denotes the canonical projection, so that the functor $p^*\colon\spsh{\Theta_{1}}_{(\infty,1)}\to\spsh{\Theta_{n+1}}_{(\infty,n+1)}$ is a left Quillen functor, and hence is homotopical.
For $k>0$, a model for the suspension functor $\mathfrak S^{k-1}\colon(\infty,1)\mathscr Cat\to(\infty,k)\mathscr Cat$ is given by the homotopical functor $\Sigma^{k-1}\colon\spsh{\Theta_1}_{(\infty,1)}\to\spsh{\Theta_{n+1}}_{(\infty,n+1)}$ obtained by iterating the construction $V[1]$ from \cite[Proposition~4.6]{rezkTheta}. Hence, a model for $\mathfrak S^{k-1}\mathscr C_1$ is the representable object
$\Sigma^{k-1}p^* N^{\textrm{Rezk}}\cC_1\cong \Theta_{n}[\cC_k]$ and a model for $\mathfrak S^{k-1}\mathscr I$ is $\Sigma^{k-1}p^* N^{\textrm{Rezk}}\cI$.

A factorization of $f\colon\Delta[1]\to \Sigma^{k-1}N^{\textrm{Rezk}}\cI$ of the form \eqref{lift}
is the trivial factorization
\[
\begin{tikzcd}
    f \colon &[-1cm] {\Sigma^{k-1}N^{\textrm{Rezk}}\cC_1} \arrow[r, hook,"{}"] & {\Sigma^{k-1}N^{\textrm{Rezk}}\cI} \arrow[r, "\simeq" swap, "{\psi}"] & \Sigma^{k-1}p^* N^{\textrm{Rezk}}\cI.
    \end{tikzcd}
\]

The following is an application of \cref{MClemma2}. See also \cite[\textsection 11.14]{rezkTheta} for other characterizations of $(\infty,n+1-k)$-equivalences.

\begin{prop}
A $k$-morphism $F\colon A\to B$ in a complete Segal $\Theta_{n+1}$-space $\mathscr C$ is an $(\infty,n+1-k)$-equivalence if and only if there is a solution to the following (strict!) lifting problem in the category $\spsh{\Theta_{n+1}}$
\[\begin{tikzcd}
\Theta_{n+1}[\cC_k]\arrow[r,"F"]\arrow[d,  "f" swap,hook ]&\mathscr C\\
\Sigma^{k-1}p^* N^{\textrm{Rezk}}\cI\arrow[ru,dashed]
\end{tikzcd}
\]
\end{prop}

\bibliographystyle{amsalpha}
\bibliography{ref}

\end{document}
